\documentclass[10pt,reqno]{amsart}
 \usepackage{amssymb,amsmath,amstext,amsthm,amsfonts}
\usepackage{euscript,graphicx}
\usepackage[ansinew]{inputenc}

\usepackage{mathpazo} 

\usepackage[colorlinks=true,backref]{hyperref}

\numberwithin{equation}{section}

\theoremstyle{plain}
\newtheorem{maintheorem}{Theorem}

\newtheorem{theorem}{Theorem}[section]
\newtheorem{proposition}[theorem]{Proposition}
\newtheorem{corollary}[theorem]{Corollary}
\newtheorem{lemma}[theorem]{Lemma}

\newtheorem{remark}[theorem]{Remark}

\newtheorem{example}{Example}
\theoremstyle{definition}

\newtheorem{conjecture}{Conjecture}

\newcommand{\RR}{{\mathbb R}}

\newcommand{\ZZ}{{\mathbb Z}}

\newcommand{\EE}{{\mathbb E}}

\newcommand{\sS}{{\mathbb S}}

\newcommand{\D}{\EuScript{D}}

\newcommand{\cP}{\EuScript{P}}
\newcommand{\cO}{\EuScript{O}}

\newcommand{\V}{\EuScript{V}}

\newcommand{\X}{\EuScript{X}}
\newcommand{\cC}{\EuScript{C}}

\newcommand{\M}{\EuScript{M}}

\newcommand{\B}{\EuScript{B}}
\newcommand{\R}{\EuScript{R}}

\newcommand{\K}{\EuScript{K}}

\newcommand{\cF}{{\mathcal F}}

\newcommand{\cU}{{\mathcal U}}

\newcommand{\cW}{{\mathcal W}}

\newcommand{\wt}{\widetilde}

\newcommand{\vfi}{{\varphi}}
\renewcommand{\epsilon}{\varepsilon}

\newcommand{\qand}{\quad\text{and}\quad}

\DeclareMathOperator{\diam}{diam}
\DeclareMathOperator{\spec}{sp}
\DeclareMathOperator{\dist}{dist}
\DeclareMathOperator{\supp}{supp}

\DeclareMathOperator{\sing}{Sing}

\DeclareMathOperator{\lip}{Lip}
\DeclareMathOperator{\close}{Closure}

\DeclareMathOperator{\graph}{Graph}

\newcommand{\m}{{\rm Leb}}

\title[Statistical stability for sectional-hyperbolic
attracting sets]{Finitely many physical measures for
  sectional-hyperbolic attracting sets and statistical
  stability}

\thanks{The author was partially supported by
  CNPq-Brazil (grant 300985/2019-3).}

\date{\today}

\author[Vitor Araujo]{Vitor Araujo}
\email{vitor.araujo.im.ufba@gmail.com}
\urladdr{https://sites.google.com/site/vdaraujo99/}
\address{Instituto de Matem\'atica e Estat\'{\i}stica,
  Universidade Federal da Bahia, Av. Ademar de Barros s/n,
  40170-110 Salvador, Brazil.}



\begin{document}

\begin{abstract} 
  We show that a sectional-hyperbolic attracting set for a
  H\"older-$C^1$ vector field admits finitely many
  physical/SRB measures whose ergodic basins cover Lebesgue
  almost all points of the basin of topological
  attraction. In addition, these physical measures depend
  continuously on the flow in the $C^1$ topology, that is,
  sectional-hyperbolic attracting sets are statistically
  stable. To prove these results we show that each
  central-unstable disk in a neighborhood of this class of
  attracting sets is eventually expanded to contain a ball
  whose inner radius is uniformly bounded away from
  zero. 
\end{abstract}

\maketitle

\tableofcontents

\section{Introduction and statements of the results}
\label{sec:introd-statem-result}

The term \emph{statistical properties} of a dynamical system
refers to the statistical behavior of {\em typical}
trajectories of the system. It is well known that this
relates to the properties of the evolution of measures by
the dynamics.  Statistical properties are often a better
object to be studied than pointwise behavior. In fact, the
future behavior of initial data can be unpredictable, but
statistical properties are often regular and their
description simpler.

Arguably one of the most influential concepts in the theory
of Dynamical Systems has been the notion of physical (or
$SRB$) measure. We say that an invariant probability measure
$\mu$ for a flow $\phi_t$ is \emph{physical} if the set
\[
B(\mu)=\left\{z\in M:
\lim_{t\to\infty}\frac{1}{t}\int_{0}^{t}\psi(\phi_s(z))\,ds=
\int\psi\, d\mu, \forall \psi\in C^0(M,\RR)\right\}
\]
has non-zero volume, with respect to any volume form on the
ambient compact manifold $M$. The set $B(\mu)$ is by
definition the \emph{basin} of $\mu$. It is assumed that
time averages of these orbits be observable if the flow
models a physical phenomenon.

The study of the existence of these special measures and
their statistical properties for uniformly hyperbolic
diffeomorphisms and flows has a long and rich history,
starting with the works of Sinai, Ruelle and Bowen
\cite{Bo75,BR75,Ru76,Ru78,Si72}. Some classes of systems
that not satisfy all the basic assumptions of uniform
hyperbolicity have been shown to possess physical measures
much more recently: sectional-hyperbolicity is a
generalization of Smale's notion of Axiom~A~\cite{Sm67} that
allows for the inclusion of equilibria (also known as
singularities or steady-states) and incorporates the
classical Lorenz attractor~\cite{Lo63} as well as the
geometric Lorenz attractors of~\cite{ABS77,GW79}.  For
three-dimensional flows, sectional hyperbolic attractors are
precisely the ones that are robustly transitive, and they
reduce to Axiom A attractors when there are no
equilibria~\cite{MPP04}.

For arbitrary dimensions this notion was established first
in~\cite{MeMor06} and the first concrete example provided
by~\cite{BPV97}. Sectional-hyperbolic attractors are
those robustly transitive attracting sets for which the flow
in a star flow in the trapping region, that is, there are no
bifurcations of singularities or periodic orbits for all
nearby dynamics (also known as ``strongly homogeneous
flows''). Again these sets reduce to Axiom A attractors if
there are no equilibria.

Sectional-hyperbolic attractors in $3$-manifolds were shown
to have a unique physical measure in \cite{APPV,AraPac2010}
and sectional-hyperbolic attracting sets have finitely many
ergodic physical measures whose basins cover a full volume
subset of a neighborhood of the attracting set;
see~\cite{Sataev2010,ArSzTr}. The study of statistical
properties of these measures is well developed program, we
mention the recent works
\cite{LMP05,HoMel,sataev2009,galapacif09,ArVar,ArGalPac,AMV15,ArMel17,ArMel18,BalMel}
among others.

Recently it was shown the existence of a unique physical
measure for sectional-hyperbolic attractors for flows in
manifolds with any finite dimension in~\cite{LeplYa17} using
the Thermodynamical Formalism and assuming certain
properties of a stable foliation in a neighborhood of the
attracting set, common to the above mentioned works in the
$3$-dimensional setting; see also~\cite{MetzMor15} for a
different proof using stochastic stability of such
attractors.

Various issues regarding the existence and smoothness of the
stable foliation in a neighborhood of sectional-hyperbolic
attracting sets are clarified in~\cite{ArMel17}; a
topological foliation always exists, and an analytic proof
of smoothness of the foliation for the classical Lorenz
attractor (and nearby attractors) is given
in~\cite{ArMel17,AMV15}. In \cite{ArMel18} sufficient
conditions are provided for these foliations to have
absolutely continuous holonomy maps, a crucial technical
feature to obtain many statistical properties in
dynamics. For higher differentiability properties of these
foliations for geometric Lorenz attractors,
see~\cite{SmaniaVidarte}.

Here we pave the way to further study of statistical
properties of sectional-hyperbolic attracting sets. We solve
the \emph{basin problem} for sectional-hyperbolic attracting
sets, that is, we show that an open dense and full measure
subset of points in a neighborhood of
these sets is exponentially asymptotic to some orbit inside
the set. More precisely: given a neighborhood $U$ of an
invariant sectional-hyperbolic attracting set $\Lambda$ of a
smooth flow $\phi_t$, there exists $K,\lambda>0$ and an open
and dense subset $W\subset U$ with full Lebesgue measure
($\m(U\setminus W)=0$) so that for
any given $y\in W$ there exists $x\in\Lambda$ satisfying
$d(\phi_ty,\phi_tx)\le Ke^{-\lambda t}$ for all $t>0$.

Moreover, coupled with recent
results from~\cite{CYZ20} on weak limits of time averages
for almost all orbits in partially hyperbolic sets with
applications to sectional-hyperbolic attracting sets, we
complement~\cite{LeplYa17} proving the existence of finitely
many ergodic physical measures for sectional-hyperbolic
attracting sets in any dimension. In addition, the basins of
these measures cover a full Lebesgue measure subset of a
neighborhood of the sectional-hyperbolic attracting set.

Having this, we use recent results from~\cite{PaYaYa} on
robust entropy expansiveness for sectional-hyperbolic
attracting sets to prove that the physical measures depend
continuously on the flow, showing that asymptotic time
averages for Lebesgue almost all points in a neighborhood of
such attracting sets are robust under small perturbations of
the dynamics. This is known as \emph{statistical stability}
and our proof provides a far-reaching extension of the
results already obtained for the $3$-flows having geometric
Lorenz attractors in~\cite{AlveSoufi} and the classical
Lorenz attractor in~\cite{bahsoun_ruziboev}.

\subsection{Preliminary definitions}
\label{sec:PH}

Let $M$ be a compact Riemannian manifold with induced
distance $d$ and volume form $\m$. Let $\X^1(M)$ be the set
of $C^1$ vector fields on $M$ and denote by $\phi^G_t$ the
flow generated by $G\in\X^1(M)$. We say that $G$ is
H\"older-$C^1$ if on any local chart the derivative $DG$ is
$\alpha$-H\"older for some fixed $0<\alpha<1$. We write
$\X^{1+}(M)$ for the vector space of all H\"older-$C^1$
vector fields over $M$.

Given a compact invariant set $\Lambda$ for $G\in \X^1(M)$,
we say that $\Lambda$ is \emph{isolated} if there exists an
open set $U\supset \Lambda$ such that
$ \Lambda =\bigcap_{t\in\R}\phi_t(U)$. If $U$ can be chosen
so that $\phi_t(U)\subset U$ for all $t>0$, then we say that
$\Lambda$ is an \emph{attracting set}.

A compact invariant set $\Lambda$ is {\em partially
  hyperbolic} if the tangent bundle over $\Lambda$ can be
written as a continuous $D\phi_t$-invariant sum
$ T_\Lambda M=E^s\oplus E^{cu}, $ where $d_s=\dim E^s_x\ge1$
and $d_{cu}=\dim E^{cu}_x\ge2$ for $x\in\Lambda$, and there
exist constants $C>0$, $\lambda\in(0,1)$ such that for all
$x \in \Lambda$, $t\ge0$, we have
\begin{itemize}
\item \emph{uniform contraction along} $E^s$:
  $ \|D\phi_t | E^s_x\| \le C \lambda^t; $ and
\item \emph{domination of the splitting}:
  $ \|D\phi_t | E^s_x\| \cdot \|D\phi_{-t} | E^{cu}_{\phi_tx}\|
  \le C \lambda^t.  $
\end{itemize}
We say that $E^s$ is the \emph{stable bundle} and $E^{cu}$
the \emph{center-unstable bundle}.  A {\em partially
  hyperbolic attracting set} is a partially hyperbolic set
that is also an attracting set.

We say that the center-unstable bundle $E^{cu}$
is \emph{sectional expanding} if for every two-dimensional
subspace $P_x\subset E^{cu}_x$,
\begin{align} \label{eq:sectional}
|\det(D\phi_t(x)\mid P_x )| \ge K  e^{\theta t}\quad\text{for all 
$x \in \Lambda$, $t\ge0$}. 
\end{align}

If $\sigma\in M$ and $G(\sigma)=0$, then $\sigma$ is called
an {\em equilibrium} or \emph{singularity} in what follows
and we denote by $\sing(G)$ the family of all such points.
An invariant set is \emph{nontrivial} if it is neither a
periodic orbit nor an equilibrium.

We say that a compact invariant set $\Lambda$ is a
\emph{sectional hyperbolic set} if $\Lambda$ is partially
hyperbolic with sectional expanding center-unstable bundle
and all equilibria in $\Lambda$ are hyperbolic.  A sectional
hyperbolic set which is also an attracting set is called a
{\em sectional hyperbolic attracting set}.

A \emph{singular hyperbolic set} is a compact invariant set
$\Lambda$ which is partially hyperbolic with volume
expanding center-unstable subbundle and all equilibria
within the set are hyperbolic. A sectional hyperbolic set is
singular hyperbolic and both notions coincide if, and only
if, $d_{cu}=2$.

\begin{remark} \label{rmk:per}
  \begin{enumerate}
  \item A sectional hyperbolic set with no equilibria is
    necessarily a \emph{hyperbolic set}, that is, the
    central unstable subbundle admits a splitting
    $E^{cu}_x=\RR\{G(x)\}\oplus E^u_x$ for all $x\in\Lambda$
    where $E^u_x$ is uniformly contracting under the time
    reversed flow; see e.g.~\cite{AraPac2010}.
  \item A sectional hyperbolic attracting set cannot contain
    isolated periodic orbits.  For otherwise such orbit must
    be a periodic sink, contradicting volume expansion.
  \end{enumerate}
\end{remark}

We recall that a subset $\Lambda \subset M$ is
\emph{transitive} if it has a full dense orbit, that is,
there exists $x\in \Lambda$ such that
$\close{\{\phi_tx:t\ge0\}}=\Lambda=
\close{\{\phi_tx:t\le0\}}$.

A nontrivial transitive sectional hyperbolic attracting set
is a \emph{sectional hyperbolic attractor}. For more details
on these notions, see e.g. \cite{AraPac2010} and references
therein.

\subsection{Statement of the results}
\label{sec:statement-results}
  
The definition of singular-hyperbolicity ensures that every
invariant probability measure supported in a
singular-hyperbolic set is a hyperbolic measure. Moreover,
if the vector field is smooth (at least H\"older-$C^1$) from
the proof of \cite[Theorem B, Section 4]{APPV} or explicitly
from \cite[Theorem 1.5]{Sataev2010}, we get that \emph{every
  singular-hyperbolic attracting set admits finitely many
  $\mu_1, \dots,\mu_k$ ergodic physical/SRB invariant
  measures; and the union of the
  ergodic basins of these measures covers a full Lebesgue
  measure subset of the topological basin of attraction of
  $\Lambda$, i.e. $\m(U\setminus\cup_{i=1}^kB(\mu_i))=0$.}

We show here that the same result is true in higher
dimensions for sectional-hyperbolic attracting sets.

\begin{maintheorem}
  \label{mthm:physectional}
  Every sectional-hyperbolic attracting set for a
  H\"older-$C^1$ vector field admits finitely many
  $\mu_1, \dots,\mu_k$ ergodic physical/SRB invariant
  probability measures. Moreover, the union of the ergodic
  basins of these measures covers a full Lebesgue measure
  subset of the topological basin of attraction of
  $\Lambda$. 
\end{maintheorem}

In \cite{LeplYa17} existence a uniqueness of the physical
measure was obtained for sectional-hyperbolic
\emph{attractors} of $C^2$ vector fields. We extend the
argument from \cite{LeplYa17} avoiding the use of a dense
orbit taking advantage of the recent results 
from~\cite{CYZ20} which hold in the $C^{1+}$ topology.


By robustness of partial hyperbolicity and sectional
expansion, given a sectional-hyperbolic attracting set
$\Lambda_G(U)=\cap_{t>0}\phi_t(U)$ with trapping region $U$,
then there exists a neighborhood $\cU\subset\X^{1+}(M)$ of
$G$ so that $U$ is a trapping region and $\Lambda_Y(U)$ is
sectional-hyperbolic for all $Y\in\cU$. It is then natural
to study the stability of the physical measures under small
perturbation of the vector field $G$.

\begin{maintheorem}
  \label{mthm:statstability}
  Let $G\in\X^{1+}(M)$ be a vector field having a trapping
  region $U$ whose attracting set
  $\Lambda_G(U)=\cap_{t>0}\phi_t(U)$ is
  sectional-hyperbolic.  Then there exists a neighborhood
  $\cU\subset\X^{1+}(M)$ of $G$ so that, for each choice of
  $G_n\in\cU$ and $\mu_n$ a physical measures for $G_n$
  supported in $U$ such that $\|G_n-G\|_{C^1}\to0$ when
  $n\nearrow\infty$, each weak$^*$ accumulation point $\mu$
  of $(\mu_n)_{n\ge1}$ is a linear convex combination of the
  ergodic physical measures of $\Lambda_G$ provided in
  Theorem~\ref{mthm:physectional}:
  \begin{align*}
    \mu\in\Phi(G)=\big\{ \sum_{i=1}^k t_i\mu_i: t_i\ge0 \qand
    \sum_{i=1}^kt_i=1\big\}.
  \end{align*}
\end{maintheorem}
In other words, the convex hull $\Phi(G)$ of the ergodic
physical measures of a sectional-hyperbolic attracting set
depends continuously on the vector field, with respect to
the $C^1$ topology of vector fields and weak$^*$ topology of
probability measures on a manifold.


Statistical stability means that time averages
$\wt{\psi}^G=\lim\frac{1}{t}\int_{0}^{t}\psi\circ\phi^G_s\,ds$
of continuous observables $\vfi:U\to\RR$ in a neighborhood
of the sectional-hyperbolic attracting sets, well-defined
Lebesgue almost everywhere in $U$, depend continuously on
the vector field $G$ generating the flow $\phi^G_t$, so that
we can assure that $|\wt{\psi}^G-\wt{\psi}^{G'}|$ is small
as long has $\|G-G'\|_{C^1}$ is small enough.

Theorem~\ref{mthm:statstability}
improves both~\cite{AlveSoufi} and~\cite{bahsoun_ruziboev}
since, although not dealing with the density of the
invariant probability of the quotient map along stable
leaves on global cross-section of the geometric Lorenz
attractor, its statement and proof applies to a much larger
family of sectional-hyperbolic attracting sets.

In particular, the attracting sets appearing as small
perturbations of singular-hyperbolic attractors as in
Morales \cite{morales04}, which must have a singular
component, are statistical stable whatever the number of
singularities involved.

We note that there are many examples of singular-hyperbolic
attracting sets, non-transitive and containing
non-Lorenz-like singularities; see
Figure~\ref{fig:singhypattracting} for an example obtained
by conveniently modifying the geometric Lorenz construction,
and many others in \cite{Morales07}. statistical
stability follows for all these examples.

\begin{figure}[h]
\centering
\includegraphics[width=6cm]{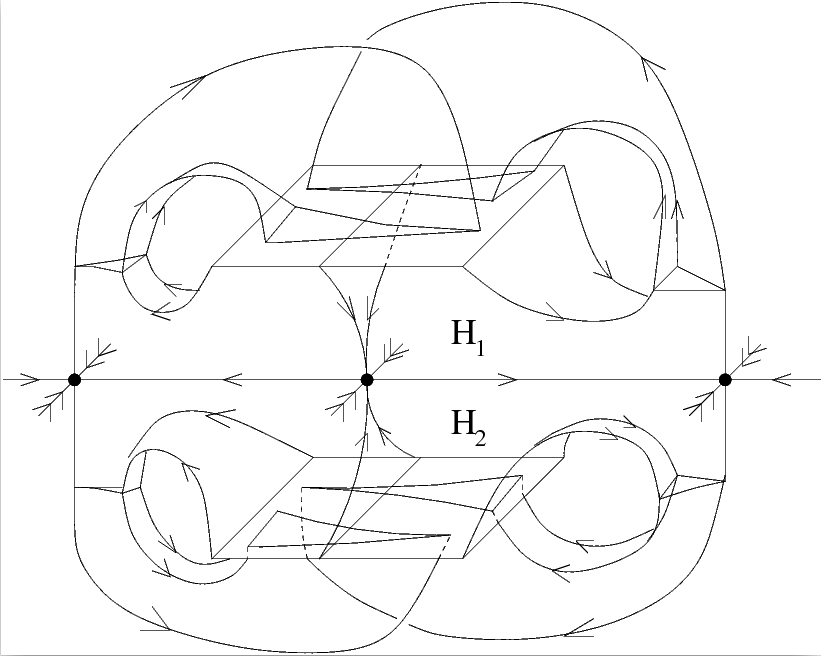}
\caption{\label{fig:singhypattracting}Example of a
  singular-hyperbolic attracting set, non-transitive (in
  fact, it is the union of two transitive sets indicated by
  $H_1,H_2$ above) and containing non-Lorenz like
  singularities.}
\end{figure}

Moreover Theorem~\ref{mthm:statstability} applies to the
multidimensional Lorenz attractor described in \cite{BPV97}
without further ado.

In addition, the open families of Lorenz-like attractors
obtained after bifurcating saddle-connections by many
authors
\cite{Ro89,Rychlik90,KKO93,DKO95,MPu97,Mo96,Ro92,Ro2000,MPS05}
are automatically endowed with statistical stability after
Theorem~\ref{mthm:statstability}, that is, in the (generic)
unfolding of double (resonant) homoclinic cycle or
saddle-connections, the physical measure for the ensuing
Lorenz-like attractors depends continuously on the
parameters.

We mention that in the preprint~\cite{MetzMor15} the authors
prove stochastic stabilty for $C^2$ \emph{transitive}
sectional-hyperbolic attracting sets which, in particular,
provides an alternative proof of the existence of a unique
SRB measure for $C^2$ sectional-hyperbolic
\emph{attractors}. This is a different kind of stability
which is in general unrelated to statistical
stability. Whereas statistical stability compares the SRB
measures of close vector fields, stochastic stability
considers random perturbations of a given vector field,
usually though a diffusion, and checks whether the
stationary measure for the randomly perturbed vector field
converges to some special invariant measure for the orinal
unperturbed vector field when the size of the diffusion
vanishes. The results strongly depend on the type of random
perturbation chosen, see e.g. \cite[Appendix D]{BDV2004}.



The proofs of Theorems~\ref{mthm:physectional}
and~\ref{mthm:statstability} use a construction of adapted
cross-sections, generalizing the one presented in the $3$-flow setting
in~\cite{APPV} and in the codimension $2$ setting in~\cite{ArMel18},
which has been used to prove many delicate statistical properties of
these flows; a similar construction (but built in a different way) of
higher-dimensional adapted cross-sections was recently proposed in the
work~\cite{CrovYang20}. This enables us to solve the \emph{basin
  problem}, as follows; see e.g. \cite{BeV00} for a similar but more
delicate instance in a highly non-uniformly hyperbolic setting.

For a periodic point $p$ of $G$ we write $\cO(p)$ for its
compact orbit $\{\phi_tp:t\in\RR\}=\{\phi_tp:t\in[0,T_p]\}$
where the minimal $T_p>0$ satisfying this is the period of
any point $q\in\cO(p)$. Moreover, we write
\begin{align*}
  W^s_x=\{y\in M: d(\phi_ty,\phi_tx)\xrightarrow[t\to+\infty]{}0\}
\end{align*}
for the \emph{stable manifold} of $x\in M$ and for a given
$\epsilon_0>0$ we write
\begin{align*}
  W^u_x(\epsilon_0)
  =
  \{y\in M: d(\phi_{-t}y,\phi_{-t}x)\le\epsilon_0, \forall
  t\ge0\}
\end{align*}
for the \emph{local unstable manifold} of $x$ of size
$\epsilon_0$. There are analogous and dual notions of local
stable manifolds and unstable manifolds.

We say that a periodic point $p$ is hyperbolic if
$D\phi_{T_p}(p):T_pM\circlearrowleft$ admits three
$D\phi_{T_p}$-invariant subspaces forming a splitting
$T_pM=E^s_p\oplus \langle G\rangle\oplus E^u_p$, where
$ \langle G\rangle=\RR\cdot G$ is an eigenspace with eigenvalue $1$;
$E^s_p$ is contracted and $E^u_p$ expanded. The (Un)Stable Manifold
Theorem ensures that $W^u_q(\epsilon_0)$ is an embedded manifold with
$T_qW^u_q(\epsilon_0)=E^u_q$ and $W^s_q$ is an immersed submanifold
with $T_qW^s_q=E^s_q$, for each $q\in\cO(p)$. We write
$W^u_{\cO(p_i)}(\epsilon_0)$ for the union
$\cup\{W^u_q(\epsilon_0):q\in\cO(p)\}$ and analogously
$W^s_{\cO(p_i)}=\cup\{W^s_q(\epsilon_0):q\in\cO(p)\}$.  For more
details on these notions from Hyperbolic Dynamics, see
e.g. \cite{PM82}.

\begin{maintheorem}
  \label{mthm:topbasinsectional}
  Let $G\in\X^1(M)$ be a vector field having a trapping region $U$
  whose attracting set $\Lambda=\cap_{t>0}\phi_t(U)$ is
  sectional-hyperbolic. Then there are $\epsilon_0>0$ and finitely
  many (hyperbolic) periodic points $p_1,\dots,p_l$ of $\Lambda$ so
  that
  \begin{align*}
    \cW^{cs}=\{W^s_x : x\in W^u_{\cO(p_i)}(\epsilon_0); i=1,\dots,l\}
  \end{align*}
  is open and dense in
$
    \cU=
    \{y\in M: d\big(\phi_t y,\Lambda\big)
    \xrightarrow[t\to+\infty]{}0 \}\supset U.
 $
 In particular, $\cup_iW^s_{\cO(p_i)}$ is dense in $\cU$.

 In addition, if $G\in\X^{1+}(M)$, then $\cW^{cs}$ contains the basin
 of any physical probability measure supported in $\Lambda$ and has
 full volume: $\m(\cU\setminus \cW^{cs})=0$.
\end{maintheorem}

Recently \cite{Yang2019} obtains a similar structure for $C^2$
sectional-hyperbolic attractors, showing that they are homoclinic
classes.

We conjecture that in this setting $\cW^{sc}$ is the
entire basin, as follows.

\begin{conjecture}
  \label{conj:fullbasin}
  For any sectional-hyperbolic attracting set $\Lambda$ the
  topological basin of attraction coincides with the family
  of stable manifolds through the points of local unstable
  leaves of finitely many periodic orbits, that is,
  $\cU=\cW^{cs}$.
\end{conjecture}

The following example shows that partially hyperbolic
attracting sets which are not sectional-hyperbolic do not
necessarily satisfy the conclusions of
Theorem~\ref{mthm:topbasinsectional}.

\begin{example}
  Consider {\em Bowen's example flow\/} (see \cite{Ta95} for
  the not very clear reason for the name) is a folklore
  example showing that Birkhoff averages need not exist
  almost everywhere.  Indeed, in the system pictured in
  Figure \ref{fig:Bowensys} time averages only exist for the
  sources $s_3$, $s_4$ and for the set of separatrixes and
  saddle equilibria
  $\Lambda=W_1\cup W_2 \cup W_3 \cup W_4 \cup \{s_1,s_2\}$,
  which is an attracting set.

\begin{figure}[htpb]
  \centering \includegraphics[width=10cm]{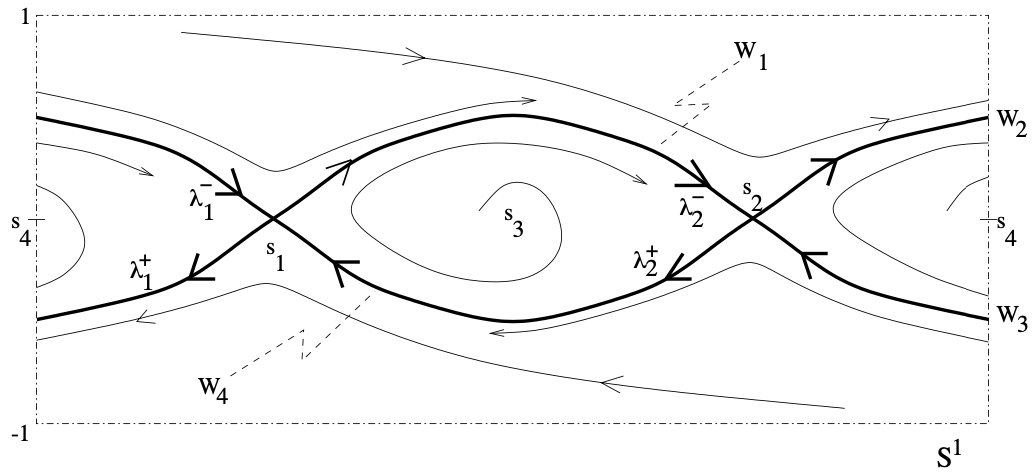}
  \caption{\label{fig:Bowensys} A sketch of Bowen's example
    flow.}
\end{figure}

The orbit under this flow $\phi_t$ of every point
$z\in \sS^1\times [-1,1] = M$ not in
$\Lambda\cup\{s_3,s_4\}$ accumulates on either side of the
separatrixes, as suggested in the figure, if we impose the
condition
$\lambda_1^{-} \lambda_2^{-} > \lambda_1^{+} \lambda_2^{+}$
on the eigenvalues of the saddle fixed points $s_1$ and
$s_2$; for more specifics on this see e.g.  \cite{Ta95} and
references therein.

Due to the very long sojourn times around $s_1$ and $s_2$,
future time averages of continuous functions $\vfi:M\to\RR$
with $\vfi(s_1)\neq\vfi(s_2)$ do not exist; see
e.g. \cite{JNY09}. However, time averages do exist for
points in $\Lambda\cup\{s_3,s_4\}$.

Hence, no point in $M\setminus(\Lambda\cup\{s_3,s_4\})$
belongs to the stable manifold of any point of $\Lambda$.
That is, we have that
$W^s(\Lambda)=M\setminus(\Lambda\cup\{s_3,s_4\})$ but the
union of the stable manifolds of the points of $\Lambda$ is
simply $\Lambda$.

To obtain an example with Lorenz-like singularities, just
multiply this system by the North-South flow on $\sS^1$ so
that the derivative of the flow at the sink in the south
pole $S$ has an eigenvalue $\lambda<0$ so that
$\lambda_i^++\lambda>0, i=1,2$; and then take the attracting
set $\Lambda\times\{S\}$ in $M\times\sS^1$.
\end{example}


\subsection{Organization of the text}
\label{sec:organization-text}

In Section~\ref{sec:preliminary-results} we present precise
statements of the main properties of sectional-hyperbolic
attracting sets together with a precise description of the
construction of a family of adapted cross-sections $\Xi$ and a
corresponding piecewise smooth and uniformly hyperbolic
global Poincar\'e return map (with singularities) on a
subset $\Xi''$ of $\Xi$, which might be of independent
interest for further work on statistical properties of
these systems.

In Section~\ref{sec:basin-problem-sectio} we consider the basin of a
$C^1$ sectional-hyperbolic attracting set $\Lambda$
proving the topological part of the statement of
Theorem~\ref{mthm:topbasinsectional}
as a consequence of showing that \emph{every center-unstable
  disk contains subdisks which are sent by arbitrarily large
  iterates of the Poincar\'e map to center-unstable disks
  with inner radius uniformly bounded away from zero
  accumulating the local unstable manifold of a hyperbolic
  periodic orbit} from a finite subset of such orbits in the
attracting set.  In
Subsection~\ref{sec:finitely-many-physic} we obtain as a
consequence of the previous result that \emph{every
  positively invariant subset of $\Lambda$ containing
  $\m$-a.e. point of a central-unstable disk must contain a
  central-unstable disk with uniform inner radius}.

This enables us to complete the proof of
Theorem~\ref{mthm:physectional} in
Subsection~\ref{sec:uniform-volume-ergod} by using and completing the
relevant steps presented in \cite{LeplYa17} together with the results
from Subsection~\ref{sec:finitely-many-physic} and the more recent
results from~\cite{CYZ20}, under the assumption that the vector field
is H\"older-$C^1$. Following this, a proof of the measure theoretic
part of the statement of Theorem~\ref{mthm:topbasinsectional} is
presented in Subsection~\ref{sec:full-volume-stable}.

Finally, we present a proof of Theorem~\ref{mthm:statstability} on
statistical stability in Section~\ref{sec:statist-stabil-secti},
coupling the previous results with robust entropy expansiveness of
sectional hyperbolic attracting sets obtained in~\cite{PaYaYa}.

\subsection*{Acknowledgement} I thank the Mathematics Department at
UFBA and CAPES-Brazil for the basic support of research activities;
and also the anonymous referee for many suggestions that helped to
improve the text.


\section{Preliminary results on
  sectional-hyperbolic attracting sets}
\label{sec:preliminary-results}

Let $G$ be a $C^1$ vector field admitting a
singular-hyperbolic attracting set $\Lambda$ with isolating
neighborhood $U$.  Given $x\in M$ we denote the
\emph{omega-limit} set
\begin{align*}
  \omega(x)=\omega_G(x)=\left\{y\in M: \exists t_n\nearrow\infty\text{  s.t.
  } \phi_{t_n}x\xrightarrow[n\to\infty]{}y \right\}
\end{align*}
and the \emph{alpha-limit} set $\alpha(x)=\omega_{-G}(x)$
which are non-empty on a compact ambient space $M$.

\subsection{Lorenz-like singularities}
\label{sec:lorenz-like-singul}

We first recall some properties of sectional-hyperbolic
attracting sets extending some results from~\cite{ArMel17,
  ArMel18} which hold for $d_{cu}\ge2$.

\begin{proposition} \label{prop:generaLorenzlike} Let
  $\Lambda$ be a sectional hyperbolic attracting set and let
  $\sigma\in\Lambda$ be an equilibrium.
  If there exists $x\in\Lambda\setminus\{\sigma\}$ so that
  $\sigma\in\omega(x)\cup\alpha(x)$, then $\sigma$ is
  \emph{generalized Lorenz-like}: that is,
  $DG(\sigma)|E^{cu}_\sigma$ has a real eigenvalue
  $\lambda^s$ and
  $\lambda^u=\inf\{\Re(\lambda):\lambda\in\spec(DG(\sigma)),
  \Re(\lambda)\ge0\}$ satisfies
  $-\lambda^u<\lambda^s<0<\lambda^u$ and so the index of
  $\sigma$ is $\dim E^s_\sigma=d_s+1$.
\end{proposition}

\begin{remark}
  \label{rmk:notLorenzlike}
  \begin{enumerate}
  \item If $\sigma\in\sing(G)\cap\Lambda$ is a generalized
    Lorenz-like singularity and $\gamma_\sigma^s$ is its
    local stable manifold, then at
    $w\in\gamma_\sigma^s\setminus\{\sigma\}$ we have
    $T_w\gamma_\sigma^s=
    E^{cs}_w=E^s_w\oplus\RR\cdot\{G(w)\}$ since
    $T\gamma_\sigma^s$ is $D\phi_t$-invariant and contains
    $G(w)$ (because $\gamma_\sigma^s$ is $\phi_t$-invariant)
    and the dimensions coincide.
  \item If an equilibrium $\sigma\in\sing(G)\cap\Lambda$ is
    not generalized Lorenz-like, then $\sigma$ is not in the
    limit set of $\Lambda\setminus\{\sigma\}$, i.e. there is
    no $x\in\Lambda\setminus\{\sigma\}$ so that
    $\sigma\in\alpha(x)\cup\omega(x)$. An example is
    provided by the pair of equilibria of the Lorenz system
    of equations away from the origin: these are saddles
    with an expanding complex eigenvalue which belong to the
    attracting set of the trapping ellipsoid already known
    to E. Lorenz; see e.g. \cite[Section 3.3]{AraPac2010}
    and references therein.
  \end{enumerate}

\end{remark}

\begin{proof}[Proof of Proposition~\ref{prop:generaLorenzlike}]
  It follows from sectional-hyperbolicity that $\sigma$ is a
  hyperbolic saddle and that at most $d_{cu}$ eigenvalues
  have positive real part.  If there are only $d_{cu}-1$ such
  eigenvalues, then the constraints on $\lambda^s$ and
  $\lambda^u$ follow from sectional expansion.

  Let $\gamma$ be the local stable manifold for $\sigma$.
  If $\sigma\in\omega(x)\cap\alpha(x)$ for some
  $x\in\Lambda\setminus\{x\}$, it remains to rule out the
  case $\dim \gamma=d-d_{cu}=d_s$.

  In this case, $T_p\gamma=E^s_p$ for all
  $p\in \gamma\cap\Lambda$ and in particular
  $G(p)\in E^s_p$.  On the one hand, $G(p)\in E^{cu}_p$ (see
  e.g.~\cite[Lemma~6.1]{AraPac2010}), so we deduce that
  $G(p)=0$ for all $p\in \gamma\cap\Lambda$ and so
  $\gamma\cap\Lambda=\{\sigma\}$.

  On the other
  hand, 
  if $\sigma\in\omega(x)$ (the case $\sigma\in\alpha(x)$ is
  analogous),
  then by the local behavior of orbits near hyperbolic
  saddles, there exists
  $p\in (\gamma\setminus\{\sigma\})\cap\omega(x)\subset
  (\gamma\setminus\{\sigma\})\cap\Lambda$ which, as we have
  seen, is impossible.
\end{proof}

\subsection{Extension of the stable bundle and
  center-unstable cone fields}
\label{sec:ext-stable-bundle}

Let $\D^k$ denote the $k$-dimensional open unit disk and let
$\mathrm{Emb}^r(\D^k,M)$ denote the set of $C^r$ embeddings
$\psi:\D^k\to M$ endowed with the $C^r$ distance. We say
that \emph{the image of any such embedding is a $C^r$
  $k$-dimensional disk}.

\begin{proposition}{\cite[Proposition~3.2, Theorem~4.2 and
    Lemma~4.8]{ArMel17}}
  \label{prop:Ws}
   Let $\Lambda$ be a
  partially hyperbolic attracting set.
  \begin{enumerate}
  \item The stable bundle $E^s$ over $\Lambda$ extends to a
    continuous uniformly contracting $D\phi_t$-invariant
    bundle $E^s$ on an open positively invariant
    neighborhood $U_0$ of $\Lambda$.
  \item There exists a constant $\lambda\in(0,1)$, such that
    \begin{enumerate}
    \item for every point $x \in U_0$ there is a $C^r$
      embedded $d_s$-dimensional disk $W^s_x\subset M$, with
      $x\in W^s_x$, such that $T_xW^s_x=E^s_x$;
      $\phi_t(W^s_x)\subset W^s_{\phi_tx}$ and
      $d(\phi_tx,\phi_ty)\le \lambda^t d(x,y)$ for all
      $y\in W^s_x$, $t\ge0$ and $n\ge1$.

    \item the disks $W^s_x$ depend continuously on $x$ in
      the $C^0$ topology: there is a continuous map
      $\gamma:U_0\to {\rm Emb}^0(\D^{d_s},M)$ such that
      $\gamma(x)(0)=x$ and $\gamma(x)(\D^{d_s})=W^s_x$.
      Moreover, there exists $L>0$ such that
      $\lip\gamma(x)\le L$ for all $x\in U_0$.

    \item the family of disks $\cF^s=\{W^s_x:x\in U_0\}$ defines a
      topological foliation $\cW^s$ of $U_0$: every $x_0\in
      U_0$ admits a neighborhood $V\subset U_0$ and a
      homeomorphism $\psi:V\to\RR^{d_s}\times\RR^{d_{cu}}$
      so that $\psi(W^s_x)=\pi_s^{-1}\{\pi_s(\psi(x))\}$ where
      $\pi^s: \RR^{d_s}\times\RR^{d_{cu}} \to\RR^{d_s}$ is
      the canonical projection.
    \end{enumerate}
  \end{enumerate}
\end{proposition}

\begin{remark}
  \label{rmk:contholo}
  For any two close enough $d_{cu}$-disks $D_1,D_2$
  contained in $U_0$ and transverse to $\cF^s$ there exists
  an open subset $\hat D_1$ of $D_1$ so that $W^s_x\cap D_2$
  is a singleton. This defines the \emph{holonomy map}
  $h:\hat D_1\to D_2, \hat D_1\ni x\mapsto W^s_x\cap D_2$
  and Proposition~\ref{prop:Ws} ensures that $h$ is
  continuous.
\end{remark}

The splitting $T_\Lambda M=E^s\oplus E^{cu}$ extends
continuously to a splitting $T_{U_0} M=E^s\oplus E^{cu}$
where $E^s$ is the invariant uniformly contracting bundle in
Proposition~\ref{prop:Ws} (however $E^{cu}$ is not invariant
in general).  Given $a>0$ and $x\in U_0$, we define the {\em
  center-unstable cone field} as
$
\cC^{cu}_x(a)=\{v= v^s+v^{cu}\in E^s_x\oplus E^{cu}_x:\|v^s\|\le a\|v^{cu}\|\}$.

\begin{proposition}
  \label{prop:Ccu}
  Let $\Lambda$ be a partially hyperbolic attracting set.
  \begin{enumerate}
  \item There exists $T_0>0$ such that for any $a>0$, after
    possibly shrinking $U_0$,
    $ D\phi_t\cdot \cC^{cu}_x(a)\subset
    \cC^{cu}_{\phi_tx}(a)$ for all $t\ge T_0$, $x\in U_0$.
  \item Let $\lambda_1\in(0,1)$ be given.  After possibly
    increasing $T_0$ and shrinking $U_0$, there exist
    constants $K,\theta>0$ such that
    $|\det(D\phi_t| P_x)|\geq K \, e^{\theta t}$ for each
    $2$-dimensional subspace $P_x\subset E^{cu}_x$ and all
    $x\in U_0$, $t\geq 0$.
  \end{enumerate}

\end{proposition}

\begin{proof} For item (1)
  see~\cite[Proposition~3.1]{ArMel17}.  Item (2) follows
  from the robustness of sectional expansion; see
  \cite[Proposition 2.10]{ArMel18} with straightforward
  adaptation to area expansion along any two-dimensional
  subspace of $E^{cu}_x$.
\end{proof}

\subsection{Global Poincar\'e map on adapted
  cross-sections}
\label{sec:global-poincare-map}

We assume that $\Lambda$ is a partially hyperbolic
attracting set and recall how to construct a piecewise
smooth Poincar\'e map $f:\Xi\to \Xi$ preserving a
contracting stable foliation $\cW^s(\Xi)$.  This largely
follows~\cite{APPV} (see also~\cite[Chapter~6]{AraPac2010})
and \cite[Section 3]{ArMel18} with slight modifications to
account for the higher dimensional set up.

We write $\rho_0>0$ for the injectivity radius of the
exponential map $\exp_z:T_zM\to M$ for all $z\in U$, so that
$\exp_z\mid B_z(0,\rho_0): B_z(0,\rho_0)\to M,
v\mapsto\exp_zv$ is a diffeomorphism with
$B_z(0,\rho_0)=\{v\in T_zM: \|v\|\le\rho_0\}$ and
$D\exp_z(0)=Id$ and also $d(z,\exp_z(v))=\|v\|$ for all
$v\in B_z(0,\rho_0)$.

\subsubsection{Construction of a global adapted
  cross-section}
\label{sec:constr-global-adapte}

Let $y\in\Lambda$ be a regular point ($G(y)\neq\vec0$).
Then there exists an open flow box $V_y\subset U_0$
containing $y$. That is, if we fix $\epsilon_0\in(0,1)$
small, then we can find a diffeomorphism
$\chi:\D^{d-1}\times(-\epsilon_0,\epsilon_0)\to V_y$ with
$\chi(0,0)=y$ such that
$\chi^{-1}\circ \phi_t\circ\chi(z,s)=(z,s+t)$.  Define the
cross-section $\Sigma_y=\chi(\D^{d-1}\times\{0\})$.

\begin{remark}\label{rmk:expsection}
  We assume that
  $\Sigma_y\subset\exp_y(B_y(0,\rho_0/3)\cap G(z)^\perp)$
  and $\|D(\exp_y)^{-1}_x\|\le2$ for all $x\in\Sigma_y$
  without loss of generality.
\end{remark}

For each $x\in\Sigma_y$, let
$W^s_x(\Sigma_y)= \bigcup_{|t|<\epsilon_0}\phi_t(W^s_x)\cap
\Sigma_y$.  This defines a topological foliation
$\cW^s(\Sigma_y)$ of $\Sigma_y$.  We can also assume that
$\Sigma_y$ is diffeomorphic to $\D^{d_{cu}-1}\times\D^{d_s}$
by reducing the size of the $\Sigma_y$ if needed.  The
stable boundary
$\partial^s\Sigma_y\cong \partial\D^{d_{cu}-1}\times
\D^{d_s}\cong \sS^{d_{cu}-2}\times\D^{d_s}$ is a regular topological
manifold homeomorphic to a cylinder of stable leaves, since
$\cW^s$ is a topological foliation; i.e. $\cong$ denotes
only the existence of a homeomorphism and the subspace
topology of $\partial^s\Sigma_y$ induced by $M$ coincides
with the manifold topology.

Let $\D_a^{d_s}$ denote the open disk of radius $a\in(0,1]$
in $\R^{d_s}$.  Define the {\em sub-cross-section}
$\Sigma_y(a)\cong \D_a^{d_{cu}-1}\times \D_a^{d_s}$, and the
corresponding sub-flow box
$V_y(a)\cong\Sigma_y(a)\times(-\epsilon_0,\epsilon_0)$
consisting of trajectories in $V_y$ which pass through
$\Sigma_y(a)$. In what follows we fix $a_0=3/4$.

For each equilibrium $\sigma\in\Lambda$, we let $V_\sigma$
be an open neighborhood of $\sigma$ on which the flow is
locally conjugated by a homeomorphism to a linear flow
(Hartman-Grobman Theorem).  Let $\gamma^s_\sigma$ and
$\gamma^u_\sigma$ denote the local stable and unstable
manifolds of $\sigma$ within $V_\sigma$; trajectories
starting in $V_\sigma$ remain in $V_\sigma$ for all future
time if and only if they lie in $\gamma^s_\sigma$.


Define $V_0=\bigcup_{\sigma\in\sing(G)\cap U} V_\sigma$.  We shrink the
neighborhoods $V_\sigma$ so that they are disjoint; 
$\Lambda\not\subset V_0$; and 
$\gamma^u_\sigma\cap\partial V_\sigma\subset V_y(a_0)$ for some
regular point $y=y(\sigma)$.

By compactness of $\Lambda$, there exists $\ell\in\ZZ^+$ and
regular points $y_1,\dots,y_\ell\in\Lambda$ such that
$\Lambda\setminus V_0 \subset \bigcup_{j=1}^\ell V_{y_j}(a_0)$.
We enlarge the set $\{y_j\}$ to include the points
$y(\sigma)$ mentioned above; adjust the positions
of the cross-sections $\Sigma_{y_j}$ if necessary to ensure that
they are disjoint; and define the global cross-section
$\Xi=\bigcup_{j=1}^\ell \Sigma_{y_j}$ and its smaller
version $\Xi(a)=\bigcup_{j=1}^\ell \Sigma_{y_j}(a)$ for each
$a\in(0,1)$.

  In what follows we modify the choices of $U_0$ and $T_0$.
  However, $V_{y_j}$, $\Sigma_{y_j}$ and $\Xi$ remain
  unchanged from now on and correspond to our current choice
  of $U_0$ and $T_0$.  All subsequent choices will be
  labeled $U_1\subset U_0$ and $T_1\ge T_0$.  In particular
  $U_1\subset V_0 \cup \bigcup_{j=1}^\ell V_{y_j}(a_0)$.
  We set $\delta_0=d(\partial\Xi,\partial\Xi(a_0))>0$ where
  $\partial\Xi(a)$ is the boundary of the submanifold
  $\Xi(a)$ of $M$, $a\in(0,1]$, and $\Xi=\Xi(1)$.

  \subsubsection{The Poincar\'e map}
  \label{sec:poincare-map}

By Proposition~\ref{prop:Ws}, for any $\delta>0$ we can
choose $T_1\ge T_0$ such that
$\diam \phi_t(W^s_x(\Sigma_{y_j}))<\delta$, for all
$x\in \Sigma_{y_j}$, $j=1,\dots,\ell$ and $t>T_1$.
We fix
  $T_1=T_1(\delta_0)$ for
  $\delta=\delta_0=d(\partial\Xi,\partial\Xi(a_0))$ in what follows
  and define
$
\Gamma_0=\{x\in\Xi:\phi_{T_1+1}(x)\in\bigcup_{\sigma\in\sing(G)\cap
  U_0}(\gamma^s_\sigma\setminus\{\sigma\})\}$ and
$\Xi'=\Xi\setminus\Gamma_0$.  If $x\in\Xi'$, then
$\phi_{T_1+1}(x)$ cannot remain inside $V_0$ so there exists
$t>T_1+1$ and $j=1,\dots,\ell$ such that
$\phi_tx\in V_{y_j}(a_0)$.  Since $\epsilon_0<1$, there
exists $t>T_1$ such that $\phi_tx\in\Sigma_{y_j}(a_0)$.

For each $\Sigma=\Sigma_{y_j}\in\Xi$, we choose a center-unstable disk
$W_\Sigma$ which crosses $\Sigma$ and is transversal to
$\cW^s(\Sigma)$, that is, every stable leaf
$W^s_x(\Sigma)$ 
intersects $W_\Sigma$ transversely at only one point, for each
$x\in\Sigma$. Then, \emph{for every given $x\in W_\Sigma\cap\Xi'$}, we
define
\begin{align}\label{eq:deftauW}
  f(x)=\phi_{\tau(x)}(x)
  \quad\text{where}\quad
 \tau(x)=\inf\left\{t>T_1:\phi_tx\in\bigcup_{j=1}^\ell
\close{\Sigma_{y_j}(a_0)}\right\}.
\end{align}
We note that by the choice of $T_1=T_1(\delta_0)$ we have
$\diam \phi_{\tau(x)}(W^s_x(\Sigma))<\delta_0$ and so the disk
$\phi_{\tau(x)}(W^s_x(\Sigma))$, although not necessarily contained in
any $\Sigma_{y_j}$, is certainly contained in some $V_{y_j}$ by
construction. Thus we can define
\begin{align}\label{eq:deftau}
  f(y)=\phi_{\tau(y)}(y)
  \quad\text{where}\quad
 \tau(y)=\inf\left\{t>T_1:\phi_ty\in\bigcup_{j=1}^\ell
\Sigma_{y_j}\right\}.
\end{align}
\emph{for each $y\in\cW^s_x(\Sigma)$}. This defines $\tau$ and $f$ on
$\Xi'$.

We define the topological foliation
$\cW^s(\Xi)=\bigcup_{j=1}^\ell \cW^s(\Sigma_{y_j})$ of $\Xi$ with
leaves $W^s_x(\Xi)$ passing through each $x\in\Xi$.  From the uniform
contraction of stable leaves together with the choice of $\delta_0$,
$T_1(\delta_0)$ and the definition of $\cW^s(\Xi)$ and flow invariance
of $\cW^s$ we obtain
\begin{proposition}{\cite[Proposition 3.4]{ArMel18}}
  \label{prop:invf}
  For big enough $T_1>T_0$,
  $f(W^s_x(\Xi))\subset W^s_{fx}(\Xi)$ for all $x\in\Xi'$.
\end{proposition}

\begin{remark}
  \label{rmk:ArMelerr}
  The definition of $\tau$ in \cite{ArMel18} is pointwise, viz.
  $ \tau(x)=\inf\{t>T_1:\phi_tx\in\bigcup_{j=1}^\ell
  \close{\Sigma_{y_j}(a_0)}\}$ and this allowed discontinuities of the
  return time and return map along stable leaves, i.e., a pair of
  indexes $i\neq j$ and of close points $w,z\in W^s_x$ so that
  $\phi_{\tau(w)}(w)\in\close{\Sigma_{y_j}(a_0)}$ and
  $\phi_{\tau(z)}(z)\in\close{\Sigma_{y_i}(a_0)}$. So the statement of
  \cite[Proposition 3.4]{ArMel18} (which is the analogous to
  Proposition~\ref{prop:invf} with $d_{cu}=2$) does not make sense in
  this set up, since we would have $f(W^s_x)$ with elements in
  distinct cross-sections.

  The definition of $\tau$ first on the points of a fixed collection of
  center-unstable leaves~\eqref{eq:deftauW}, and then the smooth
  extension to the stable leaves through each of these
  points~\eqref{eq:deftau}, provides for the crucial property stated
  in Proposition~\ref{prop:invf}.

  The rest of the features of the global Poincar\'e map $f$ stated and
  used in \cite{ArMel17,ArSzTr,ArMel18} remain valid with the same
  proofs.
\end{remark}

In this way we obtain a
piecewise $C^r$ global Poincar\'e map
$f:\Xi'\to\Xi=\bigcup_{j=1}^\ell\Sigma_j(a_0)$ with
piecewise $C^r$ roof function $\tau:\Xi'\to[T_1,\infty)$,
and deduce the following standard result.

\begin{lemma}{\cite[Lemma 3.2]{ArMel18}} \label{lem:log} If
  $\Sigma_y$ contains no equilibria
  (i.e. $\Gamma_0\cap\Sigma_y=\emptyset$), then
  $\tau\mid\Sigma_y\le T_1+2$. In general, there is $C>0$ so
  that $ \tau(x)\le -C\log\dist(x,\Gamma_0)$ for all
  $x\in\Xi'$; in particular, $\tau(x)\to\infty$ as
  $\dist(x,\Gamma_0)\to0$.
\end{lemma}

We define
$\partial^s\Xi(a_0)=\bigcup_{j=1}^\ell
\partial^s\Sigma_{y_j}(a_0)$ and
$\Gamma_1=\{x\in \Xi':fx\in\partial^s\Xi(a_0)\}$ and then
set $\Gamma=\Gamma_0\cup\Gamma_1$. Clearly
$\Gamma_0\cap\Gamma_1=\emptyset$.

\begin{lemma}
  \label{le:Gamma1}
  \begin{enumerate}
  \item $\Gamma_0$ is a $d_s$-submanifold of $\Xi$ given by a
    finite union of stable leaves
    $W_{s_i}(\Xi), i=1,\dots,k$; and
  \item $\Gamma_1$ is a regular embedded $(d-2)$-topological
    submanifold foliated by stable leaves from $\cW^s(\Xi)$
    with finitely many connected components.
  \end{enumerate}
\end{lemma}

\begin{remark}
  \label{rmk:topsubmanifold}
  Note that $\Gamma_0$ is a (smooth) submanifold of $\Xi$
  with codimension $d_{cu}-1$, so it separates $\Xi$ only if
  $d_{cu}=2$; while $\Gamma_1$ is a regular topological
  codimension $1$ submanifold of $\Xi$ and so it separates
  $\Xi$.
\end{remark}

\begin{proof}
  It is clear that $W^s_x(\Xi)\subset \Gamma$ for all
  $x\in\Gamma$, so $\Gamma$ is foliated by stable leaves.
  We claim that $\Gamma$ is precisely the set of those
  points of $\Xi$ which are sent to the boundary of $\Xi$ or
  never visit $\Xi$ in the future.

  Indeed, if $x_0\in\Xi'\setminus\Gamma_1$, then
  $fx_0=\phi_{\tau(x_0)}(x_0)\in \Sigma'$ for some
  $\Sigma'\in\Xi(a_0)=\{\Sigma_{y_j}(a_0)\}$.  For $x$ close
  to $x_0$, it follows from continuity of the flow that
  $fx\in\Sigma'$ (with $\tau(x)$ close to $\tau(x_0)$).
  Hence $x\in\Xi'\setminus\Gamma_1$ and since
  $\Xi'=\Xi\setminus\Gamma_0$, then the claim is proved and,
  moreover, $\Gamma$ is closed.

  For item (1), we note that
  $\Gamma_0\subset\Xi\cap\phi^{-1}_{[0,T_1+1]}\big(\bigcup_\sigma
  \gamma^s_\sigma\big)$ and we may assume without loss of
  generality that the above union comprises only generalized
  Lorenz-like equilibria; cf.
  Remark~\ref{rmk:notLorenzlike}(2). Hence
  $T_w\gamma_\sigma^s=E^{cs}_w$ for
  $w\in\gamma_\sigma^s\setminus\{\sigma\}$; see
  Remark~\ref{rmk:notLorenzlike}(1). Thus $\Gamma_0$ is
  contained in the transversal intersection between a
  compact $(d_s+1)$-submanifold and a compact
  $(d-1)$-manifold, so $\Gamma_0$ is a compact
  differentiable $d_s$-submanifold of $M$ and $\Xi$. In
  addition, since $\Gamma_0$ is foliated by stable leaves
  which are $d_s$-dimensional, then $\Gamma_0$ has only
  finitely many connected components in $\Xi$.

  For item (2), note that for each $x\in\Gamma_1$ we have
  that $fx\in\partial\Sigma_j(a_0)\subset\Sigma_j$. Thus
  there exists a neighborhood $W_x$ of $x$ in $\Xi$ and
  $V_{fx}$ of $fx$ in $\Sigma_j$ so that
  $f\mid W_x: W_x\to V_{fx}$ is a diffeomorphism. Hence
  $\Gamma_1\cap W_x=(f\mid W_x)^{-1}(V_x\cap
  \partial^s\Sigma(a_0))$ is homeomorphic to a
  $(d_{cu}-2+d_s)$-dimensional disk. Moreover, this shows
  that the topology of $\Gamma_1$ is the same as the
  subspace topology induced by the topology of $\Xi$. We
  conclude that $\Gamma_1$ is a regular topological
  $(d-2)$-dimensional submanifold.
  
  It remains to rule out the possibility of existence of
  infinitely many connected components
  $\Gamma_1^m, m\in\ZZ^+$ of $\Gamma_1$ in $\Xi$. Since
  $\Xi$ contains finitely many sections only, then there
  exists cross-sections $\Sigma_j,\Sigma_i$ in $\Xi$ and,
  taking a subsequence if necessary, an accumulation set
  $\wt{\Gamma}=\lim_m \Gamma_1^m$ within $\close(\Sigma_j)$
  so that $f(\Gamma_1^m)\subset\partial^s\Sigma_i(a_0)$ for
  all $m\ge1$.  By the continuity of the stable foliation,
  $\wt{\Gamma}$ is an union of stable leaves.

  We claim that the Poincar\'e times $\tau(x_m)$ for
  $x_m\in\Gamma_1^m, m\ge1$ are uniformly bounded from
  above. For otherwise the trajectory
  $\phi_{[0,\tau(x_m)]}(x_m)$ intersects $V_\sigma$ for some
  $\sigma\in\sing(G)\cap U$ and accumulates $\sigma$. Hence,
  by the local behavior of trajectories near saddles and the
  choice of the cross-sections near $V_\sigma$, we get that
  $\wt{\Gamma}\subset\Sigma_i(a_0)$ is not contained in the
  boundary of the cross-section. This contradiction proves
  the claim. Let $T$ be an upper bound for $\tau(x_m)$.

  Then, for an accumulation point $x\in\wt{\Gamma}$ of
  $(x_m)_{m\ge1}$ we have that the trajectories
  $\phi_{[0,T]}(x_m)$ converge in the $C^1$ topology (taking
  a subsequence if necessary) to a limit curve
  $\phi_{[0,T]}(x)$ and so
  $fx=\phi_{\tau(x)}(x)\in\partial^s\Sigma_i(a_0)$. Thus
  we can find neighborhoods $W_x$ of $x$ and $V_{fx}$ of
  $fx$ in $\Xi$ so that for arbitrarily large $m$ we have
  that $f\mid W_x: W_x\to V_{fx}$ is a diffeomorphism and
  $\Gamma_1\cap W_x=(f\mid
  W_x)^{-1}(V_x\cap\partial^s\Sigma_i(a_0))$, which
  contradicts the regularity of $\Gamma_1$ as topological
  submanifold.

  This concludes the proof of item (2) and the lemma.
\end{proof}

Let us set $\Xi''=\Xi(a_0)\setminus\Gamma$ from now on.  Then
$\Xi''=S_1\cup\dots\cup S_m$ for some $m\ge1$, where each
$S_i$ is a connected \emph{smooth strip}, homeomorphic to
either (i) $\D^{d_{cu}}\times \D^{d_s}$ if
$\Gamma_0\cap\close{S_i}=\emptyset$; or (ii)
$\D^{d_{cu}}\times (\D^{d_s}\setminus\{0\})$ otherwise. The
latter are \emph{singular (smooth) strips}.

We note that $f\mid S_i:S_i\to\Xi(a_0)$ is a diffeomorphism
onto its image, $\tau\mid S_i:S_i\to[T_1,\infty)$ is smooth
for each $i$, $\tau\mid S_i\le T_1+2$ on non-singular strips
$S_i$ and also on a neighborhood of
$\partial^s(S_i\cup\Gamma_0)$ for singular strips $S_i$.
The foliation $\cW^s(\Xi)$ restricts to a foliation
$\cW^s(S_i)$ on each $S_i$.

\begin{remark} \label{rmk:X} In what follows it may be
  necessary to increase $T_1$ leading to changes to $f$,
  $\tau$, $\Gamma$ and $\{S_i\}$ (and the constant $C$ in
  Lemma~\ref{lem:log}).  However, the global cross-section
  $\Xi=\bigcup\Sigma_{y_j}$ is fixed throughout the
  argument.
\end{remark}

\begin{remark}\label{rmk:existensions}
  Since $f$ sends $\Xi''$ into
  $\Xi=\Xi(1)$, there are \emph{smooth extensions}
  $\wt{f_i}: \wt{S_i}\to\Xi$ of $f\mid S_i:S_i\to\Xi$, where
  $\tilde S_i\supset \close(S_i)\setminus\Gamma_0$.
\end{remark}

\subsection{Hyperbolicity of the global Poincar\'e map}
\label{sec:hyperb-global-poinca}

We assume from now on that $\Lambda$ is a sectional
hyperbolic attracting set with $d_{cu}>2$ and proceed to
show that, for large enough $T_1>1$, the global Poincar\'e map
$f:\Xi''\to\Xi$ is \emph{piecewise} uniformly hyperbolic (with
discontinuities and singularities).

\subsubsection{Hyperbolicity at each smooth strip}
\label{sec:hyperb-at-each}

Let $S\in\{S_i\}$ be one of the smooth strips. Then there
are cross-sections $\Sigma$, $\wt\Sigma\in\Xi$ so that
$S\subset\Sigma$ and $f(\Sigma)\subset\wt\Sigma$.  The
splitting $T_{U_0}M=E^s\oplus E^{cu}$ induces the continuous
splitting $T\Sigma=E^s(\Sigma)\oplus E^u(\Sigma)$, where
$ E^s_x(\Sigma)=(E^s_x\oplus\RR\{G(x)\})\cap T_x{\Sigma} $
and $ E^u_x(\Sigma)=E^{cu}_x\cap T_x{\Sigma}$ for
$x\in\Sigma$; and analogous definitions apply to
$\wt\Sigma$.


\begin{proposition}\label{prop:secUH}
  The splitting $T\Sigma=E^s(\Sigma)\oplus E^u(\Sigma)$ is
  \begin{description}
  \item[invariant]
    $Df\cdot E^s_x(\Sigma) = E^s_{fx}(\wt\Sigma)$ for all
    $x\in S$, and
    $Df\cdot E^u_x(\Sigma) = E^u_{fx}(\wt\Sigma)$ for all
    $x\in\Lambda\cap S$.
  \item[uniformly hyperbolic] for each given
    $\lambda_1\in(0,1)$ there exists $T_1>0$ so that if
    $\inf \tau>T_1$, then
    $\|Df \mid E^s_x(\Sigma)\| \le \lambda_1$ for each
    $x\in S$; and
    $\|\big(Df \mid E^u_x(\Sigma)\big)^{-1}\| \ge
    \lambda_1^{-1}$ for all $x\in S\cap\Lambda$.
  \end{description}
  Moreover, there exists $0<\wt{\lambda_1}<\lambda_1$ so
  that, for all $x$ on a non-singular strip $S$, or for $x$
  on a neighborhood of $\partial^s(S\cup\Gamma_0)$ of a
  singular strip $S$ we have
  $\wt{\lambda_1} < \|\big(Df \mid
  E^s_x(\Sigma)\big)^{-1}\|$ and
  $\|Df \mid E^u_x(\Sigma)\|<\wt{\lambda_1}^{-1}$.
\end{proposition}

\begin{proof}
  See \cite[Proposition 4.1]{ArMel18} with straightforward
  adaptation to use area expansion along each
  two-dimensional subspaces within $E^u_x(\Sigma)$ in order
  to obtain uniform expansion; cf. \cite[Lemma
  8.25]{AraPac2010}. The last statement follows from the
  boundedness of $\tau$ on the designated domains;
  cf. Lemma~\ref{lem:log}.
\end{proof}

\subsubsection{Hyperbolicity of the extensions of the
  Poincar\'e maps at smooth strips}
\label{sec:hyperb-extens-poinca}

For a given $a>0$, $x\in\Sigma$ and $\Sigma\in\Xi$ we define
the \emph{unstable cone field} at $x$ as
$ \cC^u_x(\Sigma,a)=\{w=w^s+w^u\in E^s_x(\Sigma)\oplus
E^u_x(\Sigma): \|w^s\| \le a \|w^u\| \}.  $
\begin{remark}
  \label{rmk:globalcone}
  We assume that
  $\cC^u_x(\Sigma_y,a)\subset D(\exp_y)_{\exp_y^{-1}x}\cdot
  \cC^u_y(\Sigma_y,2a)$ for all $x\in\Sigma_y$ and each
  $\Sigma_y\in\Xi$ without loss of generality; recall
  Remark~\ref{rmk:expsection}. Consequently, letting
  $\pi^u:E^s_y(\Sigma_y)\oplus E^u_y(\Sigma_y)\to
  E^u_y(\Sigma_y)$ be the canonical projection, we get
  $\|\pi^uw\|/\|w\|\in(1-2a,1+2a)$ for all
  $w\in\cC^u_x(\Sigma_y,a)$, where we implicitly identify
  $\cC^u_x(\Sigma_y,a)$ with a subcone of
  $\cC^u_y(\Sigma_y,2a)$, for $x\in\Sigma_y$ and
  $\Sigma_y\in\Xi$.
\end{remark}

\begin{proposition} \label{prop:sec-cone} For any $a>0$,
  $\lambda_1\in(0,1)$, we can increase $T_1$ and shrink
  $U_1$ such that, if $\inf\tau>T_1$ and $x\in S$ and
  $S,S'\in\{S_i\}$ so that $fx\in S'$, then
  \begin{itemize}
  \item $Df(x)\cdot \cC^u_x(S,a) \subset \cC^u_{fx}(S',a)$;
    and
  \item 
    $\| Df(x)w\| \ge \|\pi^uDf(x)w\|\ge \lambda_1^{-1}
    \|w\|$ for all $w\in \cC^u_x(S,a)$.
  \end{itemize}
  Moreover $\| Df(x)w\| \le \wt{\lambda_1}^{-1}\|w\|$ for
  $x$ in a non-singular $S$ or $x$ in a neighborhood of
  $\partial^s(S\cup\Gamma_0)$ for a singular $S$.
\end{proposition}

\begin{proof}
  See \cite[Proposition 4.2]{ArMel18}, use $\wt{\lambda_1}$
  from Proposition~\ref{prop:secUH} and the estimate on
  $\|\pi^u\|$ from Remark~\ref{rmk:globalcone}.
\end{proof}

Considering the union of the smooth strips $S$, the previous
results shows that we obtain a global continuous uniformly
hyperbolic splitting $T\Xi''=E^s(\Xi)\oplus E^u(\Xi)$ in the
following sense.

\begin{theorem} \label{thm:global} For given $a>0$ and
  $\lambda_1\in(0,1)$ we obtain a global Poincar\'e map $f$
  so that the stable bundle $E^s(\Xi)$ and the restricted
  splitting
  $T_\Lambda\Xi''=E^s_\Lambda(\Xi)\oplus E^u_\Lambda(\Xi)$
  are $Df$-invariant; and
  $ Df\cdot \cC^u_x(\Xi,a)\subset \cC^u_{fx}(\Xi,a)$ and
  $\|\pi_u Df(x)w\|\ge \lambda_1^{-1}\|w\|$ for all
  $x\in\Xi''$ and $w\in\cC^u_x(\Xi,a)$.
\end{theorem}

\begin{remark}\label{rmk:extensions}
  The extensions $\wt{f_i}: \wt{S_i}\to\Xi$ of
  $f\mid S_i:S_i\to\Xi(a_0)$ mentioned in
  Remark~\ref{rmk:existensions} are such that on
  $\tilde S_i\supset \close(S_i)\setminus\Gamma_0$ the map $\wt{f_i}$
  behaves as $f$ in Propositions~\ref{prop:secUH}
  and~\ref{prop:sec-cone}. In particular,
  $\delta_1=d(S_i,\partial\wt{S_i}) \ge\wt{\lambda_1}\cdot
  d(\Xi(a_0),\Xi) = \wt{\lambda_1}\delta_0$.
\end{remark}

\section{The basin of sectional-hyperbolic
  attracting sets}
\label{sec:basin-problem-sectio}

Here we prove the topological part of the statement of
Theorem~\ref{mthm:topbasinsectional} using first the
following technical result. The measure theoretic part is
dealt with in the next section; see Subsection~\ref{sec:full-volume-stable}.

\begin{theorem}\label{thm:dense-stable}
  There are finitely many (hyperbolic) periodic points
  $p_1,\dots,p_l$ of $\Lambda$ in $\Xi$ so that
  $\cW^{s}=\{W^s_x(\Xi) : x\in W^u_{p_i}(\Xi);
  i=1,\dots,l\}$ is open and dense in $\Xi$ and, in
  particular, $\cup_i W^u_{p_i}(\Xi)$ is dense in $\Xi$.
\end{theorem}

This is enough to deduce the topological part of
Theorem~\ref{mthm:topbasinsectional}, since this implies
that $\cW^{cs}$ is open and dense in $\cU$.
In the rest of the section we prove
Theorem~\ref{thm:dense-stable}.

\subsection{Denseness of stable leaves of $\Lambda$ on $U$}
\label{sec:densen-stable-leaves}

\begin{proof}[Proof of Theorem~\ref{thm:dense-stable}]
  In what follows we say that a $C^1$
  $(d_{cu}-1)$-dimensional disk $D\subset\Sigma$ such that
  $T_xD\subset\cC^u_x(\Sigma,a)$ for all $x\in D$ is a
  \emph{center-unstable disk}, or just a $cu$-disk. A
  $cu$-disk $D$ is an \emph{unstable disk}, or just a
  $u$-disk, if for any given $x,y\in D$ there exists a
  sequence $\wt{f_i}:\tilde{S_i}\to\Xi$ of smooth extensions
  of $f_i=f\mid S_i$ together with a subsequence $i_k$ and
  $x_k,y_k\in\Xi$ so that
  $g_k=\wt{f_{i_k}}\circ\wt{f_{i_k-1}}
  \circ\cdots\circ\wt{f_2}\circ\wt{f_1}$ satisfies
  $g_kx_k=x, g_ky_k=y$ and
  $d(x_k,y_k)\le\lambda_1^{i_k}d(x,y)$ for all
  $k\ge1$\footnote{Note that an unstable disk is necessarily
    contained in the attracting set $\Lambda$.}.

  From Remark~\ref{rmk:globalcone}, if $S\subset\Sigma_y$
  for some $\Sigma_y\in\Xi$, then
  $\wt{D}=\exp_y^{-1}(D)=\graph(g:D_u\to E^s_y(\Sigma))$
  where $D_u=\pi_u\wt{D}\subset E^u_y(\Sigma)$ is a open
  subset of $E^u_y(\Sigma_y)$ and $g$ is a $C^1$ map such
  that $\|Dg\|\le2a$. Indeed, $D$ is transverse to
  $\cW^s(\Sigma_y)$ and each $W^s_x(\Sigma_y)$ is the graph
  of
  $\vfi_x:B(0,\rho)\cap E^s_y(\Sigma_y)\to E^u_y(\Sigma_y)$
  which is $C^1$ and depends continuously on $x$ in the
  $C^1$ topology; and the tangent space at any point of
  $\wt{D}$ is contained in $\cC^u_y(\Sigma_y,2a)$.

  We define
  $\rho(D)=\sup\{ r>0: B(x,r)\subset D_u, x\in
  E^u_y(\Sigma_y)\}$ as the \emph{inner radius} of any given
  $cu$-disk $D$.

  We use uniform expansion along center-unstable cones by
  the extension of $f$ to obtain
  \begin{proposition}\label{pr:cudiskexp}
    There exist $\rho_0>0$ and finitely many periodic points
    $p_1,\dots,p_l$ of $\Lambda$ in $\Xi$ so that any given
    center-unstable disk $D_0$ 
    contains a nested sequence
    $D_0\supset \hat D_1\supset \hat D_2\supset\dots$ of
    disks admitting
    \begin{itemize}
    \item a sequence $\wt{f_i}:\tilde{S_i}\to\Xi$ of smooth
      extensions 
      and;
    \item a subsequence $i_k$ so that
      $g_k=\wt{f_{i_k}}\circ\wt{f_{i_k-1}}
      \circ\cdots\circ\wt{f_2}\circ\wt{f_1}$ satisfies:
      \begin{align*}
      g_k\mid \hat D_k: \hat D_k\to D_k=g_k\hat
        D_k\subset\Xi\quad
        \text{is a diffeomorphism for each $k\ge1$.}
      \end{align*}
    \end{itemize}
    Moreover, $(D_k)_{k\ge1}$ accumulates a $u$-disk $D$ in
    the $C^1$ topology which
    \begin{itemize}
    \item contains the local unstable manifold $W^u_q(\Xi)$
      with respect to $f$ of a point of $q$ of the orbit of
      $p_i$ for some $i\in\{1,\ldots,l\}$; and
    \item whose inner radius is uniformly bounded away from
      zero: $\rho(D)\ge\rho_0$.
    \end{itemize}
  \end{proposition}
  
  We prove Proposition~\ref{pr:cudiskexp} in the next
  subsection.  Since $\cW^s(\Xi)$ is transversal to any
  $cu$-disk and the nested disks $\hat D_k$ with vanishing
  diameter intersect in a unique point
  $r\in D_0\cap\Lambda$, then this shows that \emph{every
    center-unstable disk in any smooth strip contains the
    transversal intersection of the stable manifold of all
    points in the local unstable manifold of a periodic
    point of $\Lambda$}, completing the proof of
  Theorem~\ref{thm:dense-stable}.
\end{proof}

\subsection{Local uniform expansion of $cu$-disks}
\label{sec:local-uniform-expans}

Here we fix a $cu$-disk $D_0$ in $S\in\{S_i\}$
and prove Proposition~\ref{pr:cudiskexp}. 

We obtain by induction a sequence of disks $D_n, n\ge0$ in
$\Xi$ as follows.  First, the inner radius of any $cu$-disk
contained in a smooth strip $\wt{S}$ is uniformly expanded
by the global Poincar\'e map.

\begin{lemma}
  \label{le:imagecudisk}
  If $\lambda_2\in(0,1)$ satisfies
  $2\lambda_1<\lambda_2(1-2a)$ and $D$ is a $cu$-disk
  contained in some extension $\wt{S}$ of a smooth strip
  $S\in\{S_i\}$, then
  \begin{align*}
    \rho(\wt{f}D)\ge\lambda_2^{-1}\rho(D) \qand
    (1-2a)\rho(\wt{f}D)\le2\diam(\wt{f}D)\le(1+2a)\rho(\wt{f}D),
  \end{align*}
  where $\wt{f}:\tilde S\to\Xi$ is the extension of
  $f\mid S:S\to\Xi(a_0)$.
\end{lemma}

\begin{proof}
  Let $S\subset\Sigma_y\in\Xi$.  From
  Remark~\ref{rmk:extensions}, $\wt{f}D$ is a $cu$-disk
  contained in some $\Sigma_{y'}\in\Xi$ and we can write
  $\exp_{y'}^{-1}(\wt{f}D)=\graph(g:D_u^1\to
  E^s_{y'}(\Sigma_{y'}))$ where
  $D_u^1\subset E_{y'}^u(\Sigma_{y'})$ is an open
  subset. Then for a ball $B(x',r)\subset D_u^1$ and $C^1$
  curve
  $\gamma_1:(I,0,1)\to (\close{B(x',r)},x',\partial
  B(x',r))$ there exists a unique curve
  $\gamma:I\to D_u=\pi_u\exp_y^{-1}D$ such that
  $\gamma_1(s)=\pi_u\wt{f}\exp_y(\gamma(s)+g_1\gamma(s))$,
  where $s\in I=[0,1]$. By Theorem~\ref{thm:global} and
  Remark~\ref{rmk:extensions} together with the choice of
  $\Sigma_{y},\Sigma_{y'}$ in Remark~\ref{rmk:expsection}
  \begin{align*}
    \|\dot\gamma_1(s)\|
    &=
    \big\|\pi_uD\wt{f}\cdot D\exp_y
    \big(\dot\gamma(s)+Dg_1(\gamma(s))\cdot\dot\gamma(s)\big)\big\|
    \\
    &\ge
    \lambda_1^{-1}\big\|D\exp_y
      \big(\dot\gamma(s)+Dg_1(\gamma(s))\cdot\dot\gamma(s)\big)\big\|
      \ge
    \frac{\lambda_1^{-1}}2\cdot(1-2a)
    \|\dot\gamma(s)\|.
  \end{align*}
  Then the bound on the inner radius follows by the choice
  of $\lambda_2$, since $\gamma_1$ is any curve joining
  $\gamma_1(0)=x'$ to the boundary
  $\gamma_1(1)\in\partial B(x',r)$ inside $D_u^1$. For the
  diameter, note that
  $\|u-v\|(1-2a)\le\|u+g_1u-(v+g_1v)\|\le(1+2a)\|u-v\|$ for
  all $u,v\in D_u^1$ and account the effect of $\exp_{y'}$
  on distances, cf. Remark~\ref{rmk:expsection}.
\end{proof}

We let $\lambda_2$ be as in the statement of
Lemma~\ref{le:imagecudisk} in what follows; fix
$\lambda_2<a_1<1$ and assume without loss of generality that
$a_1\lambda_2^{-1}>5$. We assume that $cu$-disks
$D_0,\dots, D_n$ have already been obtained so that there
are smooth strips $S_0,\dots,S_n$ satisfying
$D_i\subset\wt{S_i}\subset\Sigma_{y_i}$ and
$D_{i+1}\subset \wt{f_i}D_i$, $i=0,\dots,n-1$.

Letting $D_n=\exp_{y_n}\graph(g_n)$ we consider the balls
$\B=\{B(x,a_1\rho(D_n))\subset \pi_u\exp_{y_n}^{-1}D_n\}$
and corresponding disks
$\D=\{D=\exp_{y_n}\graph(g_n\mid B), B\in\B\}$. We set
$\hat\D=\{D\in\D: \exists S, D\cap\partial^s
S\neq\emptyset\}$ and
$\hat\D_\sigma=\{D\in\D: D\cap\Gamma_0\neq\emptyset\}$.
Then we have the following cases.
\begin{enumerate}
\item If $\D\not\subseteq\hat\D\cup\hat\D_\sigma$, then we
  choose some $D\in\D\setminus(\hat\D\cup\hat\D_\sigma)$.
  There exists a smooth section $S$ so that $D\subset S$ and
  we reset $D_n=D$ and define
  $D_{n+1}=fD_n=(f\mid S)(D_n)\subset\Xi(a_0)$.
\item Otherwise: either $\hat\D\neq\emptyset$ or
  $\hat\D_\sigma\neq\emptyset$.
  \begin{enumerate}
  \item If $\hat\D_\sigma\neq\emptyset$, then we choose
    $D\in\hat\D_\sigma$ and $B\subset\pi_u\exp_{y_n}^{-1}D$
    a ball of radius $a_1\rho(D_n)/4$ so that, resetting
    $D_n=\exp_{y_n}B$, we have
    \begin{align*}
      d(D_n,\Gamma_0)&>(1-2a)\rho(D_n) \qand
      \\
    \rho(D_n)&=\rho(D)/4>a_1\lambda_2^{-1} \rho(D_{n-1})/4 >
    (5/4)\rho(D_{n-1}).
    \end{align*}
We then define $D_{n+1}=fD_n$.
\item Otherwise, we have $\hat\D\neq\emptyset=\hat\D_\sigma$
  and consider the subfamily
  $\wt{\D}=\{D\in\hat\D: \exists S, D\cap\partial^s
  S\neq\emptyset\neq D\cap\partial^s\wt{S}\}$ of those disks
  which intersect both $\partial^s S$ and $\partial^s\wt{S}$
  for some $S$.
  \begin{enumerate}
  \item If $\wt{\D}=\emptyset$, then we choose some
    $D\in\hat\D$ and $S$ so that
    $D\cap\partial^s S\neq\emptyset$; reset $D_n=D$; and
    define $D_{n+1}=\wt{f}D_n$, where $\wt{f}$ is the
    extension of $f\mid S$ to $\wt{S}$.
  \item Otherwise, we choose $D\in\wt{\D}$.  There exists a
    subdisk $\hat D\subset D$ such that
    $\hat D\subset\wt{S}$ and
    $\rho(\hat D)\ge a_0\delta_1/2$ by
    Remark~\ref{rmk:extensions} and definition of $\delta_1$
    and $\hat D$. We reset $D_n=\hat D$ and define
    $D_{n+1}=\wt{f}D_n$, with $\wt{f}$ denoting the
    extension of $f\mid S$ to $\wt{S}$.
  \end{enumerate}
\end{enumerate}
\end{enumerate}

This complete the inductive step of the construction of a
sequence $(D_n)_{n\ge0}$ of $cu$-disks in $\Xi$.
Lemma~\ref{le:imagecudisk} ensures that
$\rho(D_{n+1})\ge (a_2\lambda_2^{-1}/4)\rho(D_n)$ and
$a_1\lambda_2^{-1}/4>5/4>1$ by the choice of $a_1$.

Since $\diam S<\diam \wt{S}$ is bounded by a uniform
constant for all smooth strips $S\in\{S_i\}$, the expansion
of the inner radius implies that the induction cannot go
through cases (1), (2a) or (2b-i) above consecutively
infinitely many times.  \emph{Each time we are in case
  (2b-ii) we restart the algorithm choosing a subdisk of
  $D_{n+1}$ with half the inner radius.} 

We conclude that, starting with any disk $D_0$ as above, we
obtain a subsequence $n_k\nearrow\infty$ so that $D_{n_k}$
is in case (2b-ii) and $\rho(D_{n_k})>a_0\delta_1/2$ for all
$k\ge1$. Moreover, there exists a subdisk $D'_{n_{k-1}}$
with $2\rho(D'_{n_{k-1}})=\rho(D_{n_k})$ and an iterate
$f^{m_k}$ so that
$f^{m_k}\mid D'_{n_{k-1}} : D'_{n_{k-1}} \to D_{n_k}$ is a
diffeomorphism. In addition, the uniform bound on the
diameter also ensures that $m_k=n_k-n_{k-1}$ is bounded:
$m_k\le \bar m$.

Finally, since $\Xi$ contains finitely many cross-sections,
we can assume without loss of generality that
$D_{n_k}\subset\Sigma_y\in\Xi$ for (possibly a subsequence
of) all $k$. This is a sequence of graphs of $C^1$ functions
with uniformly bounded derivative and domains given by balls
with radius uniformly bounded away from zero. It follows
that there exists a subsequence of such disks uniformly
converging to a $cu$-disk $D$ in the $C^1$-topology.

In particular, for big enough $k$ we have that the stable
holonomy map $h:D'_{n_{k-1}}\to D_{n_k}$ within $\Sigma$ is
well-defined by the choice of $D'_{n_{k-1}}$ with half of
the inner radius of $D_{n_k}$ and, moreover, is a continuous
map; see Remark~\ref{rmk:contholo}. Hence
$g=h\circ(f^{m_k}\mid D_{n_k})^{-1}: D_{n_k}\to D_{n_k}$ has
a fixed point by Brower's Fixed Point Theorem, that is,
there exists a stable leaf satisfying
$f^{m_k}W^s_x(\Xi)\subset W^s_x(\Xi)$. The contraction of
stable leaves implies the existence of a fixed point $p$ of
$f^{m_k}\mid \Sigma$ which is a periodic point for the flow
whose stable manifold crosses $D_{n_k}$. Since we can take
$k$ as big as needed, we obtain that $W^s_p(\Sigma)$
transversely crosses $D$ also.

To complete the proof, since $D_{n+1}\subset \wt{f_n}D_n$ by
construction, if we set
$g_n=\wt{f_n}\circ\cdots\circ\wt{f_1}, n\ge1$, then we can
find $\hat D_{n+1}\subset D_0$ so that $g_n\hat D_n=D_n$ and
$D_{n+1}\subset D_n$, $n\ge1$. Since $D_{n_k}\to D$
uniformly as graphs, for $x,y\in D$ there are
$\wt{x_k},\wt{y_k}\in \hat D_{n_k}$ such that
$(g_{n_k}\wt{x_k},g_{n_k}\wt{y_k})\to(x,y)$.  By uniform
expansion on $cu$-disks, for any given $i\ge1$ we get
$d(g_{n_k-i}\wt{x_k},g_{n_k-i}\wt{y_k})\le
\lambda_1^id(g_{n_k}\wt{x_k},g_{n_k}\wt{y_k})$ for all
$k\ge1$. Thus, for an accumulation pair $(x_i,y_i)$ of
$(g_{n_k-i}\wt{x_k},g_{n_k-i}\wt{y_k})$ and sequence
$g^i=\wt{f_i}\circ\cdots\circ\wt{f_0}$ of
$\wt{f_{n_k-i}}\circ\dots\circ\wt{f_{n_k}}$ as
$k\nearrow\infty$, we get $(g^ix_i,g^iy_i)=(x,y)$ and
$d(x_i,y_i)\le\lambda_1^id(x,y)$. Hence $D$ is a $u$-disk as
claimed.

In addition, by the Inclination Lemma (``$\lambda$-lemma''
\cite[Chap. 2, §7]{PM82}), we can assume without loss of
generality that $n_k$ are multiples of $m_k$ for all large
enough $k$ and that $D\subset W^u_p$ is a piece of the local
unstable manifold of the hyperbolic periodic orbit $p$ of
$f$, whose period is a divisor of $m_k$. In particular, $p$
is a hyperbolic periodic orbit for the flow whose period
$\tau_p$ is bounded: $\tau_p\le T=T(\bar m)$.

This ensures the uniform size of the unstable
manifold of the periodic orbits obtained by the previous
algorithm and, moreover, since their period is bounded, that
the possible periodic orbits are finite in number, by
hyperbolicity and compactness.

This completes the proof of Proposition~\ref{pr:cudiskexp}.


\section{Finitely many ergodic physical measures for
  sectional hyperbolic attracting sets}
\label{sec:finitely-many-ergodi}

Here we prove Theorem~\ref{mthm:physectional}. We first
obtain an auxiliary result consequence of the previous
arguments on $cu$-disks contained in adapted cross-sections.

\subsection{Uniformly center-unstable size of invariant
  subsets}
\label{sec:finitely-many-physic}

We prepare the proof of Theorem~\ref{mthm:physectional}
obtaining a result on uniform size of positively
flow-invariant subsets along the center-unstable direction.

We say that a $d_{cu}$-dimensional $C^1$ disk $D_0\subset U$
is a \emph{$cu$-disk} if
$T_xD_0\subset\cC^{cu}_x(a)$ for all $x\in D$ (observe
that such $D$ is not contained in any cross-section
$\Sigma\in\Xi$).

\begin{proposition}
  \label{pr:cusize}
  For a sectional hyperbolic attracting set $\Lambda$ of a
  $C^1$ vector field $G$, there exists $\delta>0$ so that,
  given a positively $G$-invariant subset $E\subset\Lambda$
  having a $cu$-disk $D$ such that $D\cap E$ has full
  Lebesgue induced measure in $D$, then there exists a
  $cu$-disk $\wt{D}$ whose inner radius is larger than
  $\delta$ and such that $\wt{D}\cap E$ has full
  Lebesgue induced measure in $\wt{D}$. Moreover, there
  exists $i\in\{1,\dots,l\}$ so that for any $\epsilon>0$ we
  can find $\wt{D}$ as above $\epsilon$-$C^1$-close to the
  local unstable manifold of $\cO(p_i)$, where the periodic
  point $p_i$ is given by Proposition~\ref{pr:cudiskexp}.
\end{proposition}

\begin{proof}
  This is a consequence of
  Proposition~\ref{pr:cudiskexp}. Indeed, if
  $E\subset\Lambda$ and $D$ are as stated, then we project
  $D$ into $D_0$ through the flow to the nearest
  cross-section, that is, for any $x\in D$ we consider
  $t(x)=\inf\{t>0: \phi_tx\in\Xi(a_0)\}$ and
  $p(x)=\phi_{t(x)}x, x\in D$.

    \emph{We claim that $p(D)$ contains a $cu$-disk $D_0$
      inside some $\Sigma\in\Xi$ and moreover $E\cap D_0$
      has full Lebesgue induced measure in $D_0$.}

    Assuming this claim, then $\hat D_k\cap E$ also has full
    Lebesgue induced measure in $\hat D_k$ for each of the
    disks $\hat D_k\subset D_0$ provided by
    Proposition~\ref{pr:cudiskexp}. Moreover, since the
    Poincar\'e map $f$ is a piecewise $C^1$ diffeomorphism
    as well as its extensions, then $D_k=g_k\hat D_k$ is
    such that $D_k\cap E$ also has full Lebesgue induced
    measure in $D_k$ by invariance of $E$ under all
    transformations $\phi_t, t\in\RR$. The statement of
    Proposition~\ref{pr:cusize} follows since, by
    construction, (i) the $cu$-disks $D_k$ have inner radius
    larger than some $\delta>0$ inside $\Sigma$; (ii) fixing
    some $k\ge1$ we have that
    $\wt{D}=\phi_{[-\delta,\delta]}(D_k)$ is a
    $d_{cu}$-dimensional center-unstable disk for the flow
    of $G$ with inner radius bounded away from zero; and
    (iii) by smoothness of the flow and invariance of $E$ we
    have that
    $\wt{D}\cap E=\phi_{[-\delta,\delta]}(D_k\cap E)$ also
    has full Lebesgue induced measure inside
    $\phi_{[-\delta,\delta]}(D_k)$. Moreover, the
    $C^1$-closeness to the local unstable manifold of one of
    the hyperbolic periodic orbits provided by
    Proposition~\ref{pr:cudiskexp} follows, using the
    transversal intersection of the stable manifold of this
    orbit with $\wt{D}$ together with the Inclination Lemma.

    We are left to prove the claim. Since $D\subset U$ we
    have $t(x)<\infty$ for all $x\in D$ and we fix
    $x_0\in D$ and $y_0=p(x_0)\in\Sigma$, for some
    adapted cross-section $\Sigma\in\Xi(a_0)$ in what
    follows, where we assume without loss of generality that
    $x_0$ is not a singularity.

    We take a cross-section $S$ to $G$ at $x_0$ and note
    that since $D$ is a $cu$-disk for the flow, then there
    exists a neighborhood $V$ of $x_0$ in $M$ such that (i)
    $p(V)\subset \Sigma$ and (ii) $S$ is transversal
    to $D\cap V$.  So $D_S=S\cap D\cap V$ is a submanifold
    of $M$ of codimension $1+d_s$. Hence, $D_S$ is a
    submanifold of $S$ of dimension $d_{cu}-1$ and a
    $cu$-disk inside $S$, that is,
    $T_xD_S\subset \cC^{cu}_x(a,S)$ according to the
    definition of the induced center-unstable cone fields on
    a cross-section $S$. Consequently, $p(D_S)$ is a
    $cu$-disk inside $\Sigma$ and contained in $p(D)$. We
    are left to show that $E$ has full Lebesgue induced
    measure in $p(D_S)$.

    We now conveniently choose coordinates on a local chart
    of $M$ at $V$ so that $S=\RR^{d-1}\times\{0\}$,
    $G(x_0)= (0,\dots,0,1)$ and
    $E^s=\RR^{d_s}\times\{0^{d_{cu}}\},
    E^{cu}_{x_0}=\{0^{d_s}\}\times\RR^{d_{cu}}$, and also
    $D\cap V$ is the graph of a $C^1$ map
    $\vfi:\RR^{d_{cu}}\to\RR^{d_s}$. Since
    $\Phi:\{0^{d_s}\}\times\RR^{d_{cu}}\to D\cap V, u\mapsto
    (\vfi u,u)$ is a $C^1$ diffeomorphism and
    $E\cap D\cap V$ has full Lebesgue induced measure in
    $D\cap V$, then $\wt{E}=\Phi^{-1}(E\cap D\cap V)$ has full
    Lebesgue measure in
    $\{0^{d_s}\}\times\RR^{d_{cu}}$.

    However
    $D_S=\Phi(\{0^{d_s}\}\times\{\RR^{d_{cu}-1}\times\{0\})$
    does not necessarily intersect $E$ in a full Lebesgue
    induced measure subset. But Fubbini's Theorem ensures
    that
    $\wt{E}\cap\{0^{d_s}\}\times\RR^{d_{cu}-1}\times\{t\}$
    has full Lebesgue measure for Lebesgue almost every
    $t\in\RR$.

    Thus we can choose $t$ as close to $0$ as needed so that
    $S_t=\RR^{d-1}\times\{t\}$ is a cross-section to $G$;
    $D_t=S_t\cap D\cap V$ is a $cu$-disk inside $S_t$ and
    $E\cap D_t$ has full Lebesgue induced measure in $D_t$.
    Moreover, we also have that
    $p(D_t)\subset p(D)\subset \Sigma$ is a $cu$-disk inside
    $\Sigma$ and $p(D_t\cap E)$ has full Lebesgue induced
    measure in $p(D_t)$, since $p\mid D_t: D_t\to p(D_t)$ is
    a diffeomorphism as smooth as $G$.

    This completes the proof of the claim with $D_0=p(D_t)$
    and Proposition~\ref{pr:cusize} follows.
\end{proof}

\subsection{Uniform volume of ergodic basis of physical
  measures}
\label{sec:uniform-volume-ergod}

We now extend the steps presented in \cite{LeplYa17}
together with Proposition~\ref{pr:cusize} and the following
result.

\begin{theorem}\cite[Appendix: Corollary B.1\& Theorem
  I]{CYZ20}  \label{thm:CYZ}
  A $C^1$ vector field having a sectional hyperbolic
  attracting set $\Lambda$ supports an SRB measure.  More
  precisely, for Lebesgue almost every point $x$ in the
  trapping region of $\Lambda$, any weak$^*$ limit measure of the
  family
  $\left(T^{-1}\int_0^T\delta_{\phi_tx}\,dt\right)_{T>0}$
  is an SRB measure.  Moreover, if the vector field is
  H\"older-$C^1$, then each limit measure is a physical
  measure.
\end{theorem}

The above result states that any weak$^*$ accumulation point
$\mu$ of the empirical measures along the orbit of a
Lebesgue generic point in $U$ is an equilibrium state for
the logarithm of the center-unstable Jacobian, that is
\begin{align}\label{eq:srb}
  h_\mu(\phi_1)=\int \log |\det D\phi_1\mid E^{cu}| \, d\mu > 0
\end{align}
the positiveness being a consequence of
sectional-hyperbolicity.

Moreover, \emph{if the flow is H\"older-$C^1$, then this SRB measure
  is also a physical measure} since its support contains the (Pesin)
unstable manifold through $\mu$-a.e. point and the stable foliation is
absolutely continuous\footnote{This is a consequence of the partial
  hyperbolicity of the attracting set, that the vector field is
  H\"older-$C^1$ and Proposition~\ref{prop:Ws}; this can be seen
  adapting known arguments from~\cite{PS72,Pesin2004}.}, following
standard geometric and ergodic arguments; see e.g.~\cite[Sections
2\&3]{LeplYa17} and the proof of \cite[Theorem I]{CYZ20}. In
particular, the center unstable manifold $W^{cu}_x$ through
$\mu$-a.e. $x$ is a $cu$-disk contained in the attracting set
$\Lambda$.

\begin{remark}
  \label{rmk:perptsupp}
  From Proposition~\ref{pr:cusize}, since the support of any SRB
  measure $\mu$ is a forward invariant closed subset and contains a
  $cu$-disk, then there exists a periodic orbit $\cO(p_i)$ contained
  in $\supp\mu$ for some $i\in\{1,\dots,l\}$. The stable leaves
  through the points of the local unstable manifold
  $W^{cu}_{\cO(p_i)}(\epsilon_0)$, for all small enough
  $\epsilon_0>0$, intersect each center-unstable manifold through
  $\mu$-a.e. point in an open subset, which contains $\mu$-generic
  points by the absolute continuity property of the stable
  foliation. In particular, this shows that no such periodic orbit can
  be shared by distinct SRB measures of a sectional-hyperbolic
  attracting subset.
\end{remark}

\begin{proof}[Proof of Theorem~\ref{mthm:physectional}]
  From Theorem~\ref{thm:CYZ} we have that any sectional hyperbolic
  attracing set for a $C^1$ flow admits some physical/SRB probability
  measure $\mu$ which we can assume, without loss of generality, to be
  ergodic. Indeed, using ergodic decomposition, by Ruelle's
  Inequality~\cite{Man87} we have
  $ h_\mu(\phi_1) \le\int \log |\det D\phi_1\mid E^{cu}| \, d\mu $ and
  so if $\mu$ satisfies~\eqref{eq:srb}, then each ergodic component of
  $\mu$ also satisfies~\eqref{eq:srb}

  Now we use that the ergodic basin $B(\mu)$ of $\mu$
  contains a full Lebesgue measure subset of some
  center-unstable disk $D_0$ inside the sectional-hyperbolic
  attracting set together with Proposition~\ref{pr:cusize}.

  \begin{corollary}
  \label{cor:finitevolphys}
  Every sectional hyperbolic attracting set $\Lambda$ for a
  H\"older-$C^1$ vector field admits $\epsilon_0>0$ so that
  the volume of the ergodic basin $B(\mu)$ of any ergodic
  $SRB$ measure $\mu$ supported in $\Lambda$ is uniformly
  bounded away from zero: $\m(B(\mu))\ge\epsilon_0$.
\end{corollary}

\begin{proof}
  By assumption, $\mu$ is an ergodic $SRB$-measure and, as
  explained above, in our setting
  the stable holonomies are absolutely continuous. Then
  by~\cite[Lemma 3.2]{LeplYa17} we have that there exists a
  open subset $V$ of the basin of attraction of $\Lambda$ so
  that $\m$-a.e. $x\in V$ is $\mu$-generic, that is,
  $\m(V\setminus B(\mu))=0$.

  Hence there exists a $cu$-disk $D_0\subset V$ such that
  $D_0\cap B(\mu)$ has full Lebesgue induced measure in
  $D_0$. Proposition~\ref{pr:cusize} implies that the
  positively invariant subset $B(\mu)$ contains a $cu$-disk
  $D$ with $\rho(D)\ge\delta$ for some uniform $\delta>0$
  depending only on $\Lambda$. The same proof of \cite[Lemma
  3.2]{LeplYa17}, using the uniform size of local stable
  leaves of $\cW^s$ and the angle between $E^s_x$ and
  $E^{cu}_x$ at $x\in D$ uniformly bounded away from zero
  (due to domination), implies that the set
  $W=\bigcup\{W^s_x:x\in D\}$ is open, diffeomorphic to a
  cylinder $D\times\D^{d_s}$ of uniform height. So
  $\m(W)\ge\epsilon_0$ for some uniform $\epsilon_0>0$. In
  addition, $\m$-a.e $x\in W$ belongs to $B(\mu)$ by the
  absolute continuity of the stable foliation.
\end{proof}

We are now ready to complete the proof of
Theorem~\ref{mthm:physectional}: let $U$ be a trapping
region for $\Lambda$. If $\m(U\setminus B(\mu))=0$, then
$\mu$ is the unique physical/SRB measure supported in
$\Lambda$. Otherwise, let $\mu_1=\mu$ and since
$U_1=U\setminus B(\mu_1)$ is such that $\m(U_1)>0$ we can
use \cite[Theorem I]{CYZ20} to ensure that $\m$-a.e.
$x\in U_1$ belongs to the ergodic basin of some SRB measure
$\mu_2\neq\mu_1$. This measure $\mu_2$ is a physical
measure, satisfies $\m(B(\mu_2))>\delta>0$ by
Corollary~\ref{cor:finitevolphys} and $B(\mu_1)\cap
B(\mu_2)=\emptyset$ and $B(\mu_1)\cup B(\mu_2)\subset U$.

Again, if $\m\big(U\setminus(B(\mu_1)\cup B(\mu_2))\big)=0$,
then $\Lambda$ supports exactly the pair $\mu_1,\mu_2$ of
ergodic physical measures whose ergodic basins cover the
topological basin of $\Lambda$ except perhaps a Lebesgue
zero subset. Otherwise $U_2=U\setminus(B(\mu_1)\cup
B(\mu_2))$ is such that $\m(U_2)>0$ and we can repeat the
argument.

Since the ergodic basins of distinct ergodic physical
probability measures are disjoint subsets of the trapping
region $U$ which has finite volume, and each ergodic basin
has a minimum volume bounded away from zero, this inductive
process stops with finitely many $\mu_1,\dots,\mu_k$ ergodic
physical/SRB measures supported on $\Lambda$, whose basis
cover the trapping region $U$, $\m\bmod0$. This completes
the proof of Theorem~\ref{mthm:physectional}.
\end{proof}

\subsection{Full volume of stable leaves in the topological
  basin}
\label{sec:full-volume-stable}

Here we prove the measure-theoretic part of the statement of
Theorem~\ref{mthm:topbasinsectional}.

Let $\mu$ be an ergodic SRB measure supported in $\Lambda$ for a
H\"older-$C^1$ vector field, that is, a physical measure.  Let
$\cO(p_i)$ be the hyperbolic periodic orbit contained in $\supp\mu$;
see Remark~\ref{rmk:perptsupp}. Let $V$ be an open neighborhood of
$\cO(p)$ and $\vfi:M\to\RR$ a non-negative continuous function
supported in $V$, so that $\mu(\vfi)=\int\vfi\,d\mu>0$. Hence
\begin{align*}
  \lim_{T\to\infty}\frac1T\int_0^T\vfi(\phi_tx)\, dt =
  \mu(\vfi)>0, \quad\text{for each $x\in B(\mu)$}.
\end{align*}
Thus $\vfi(\phi_tx)>0$ for some $t>0$ and so $\phi_tx\in V$.
This ensures that $x\in W^s_y$ for some
$y\in W^{cu}_{\cO(p_i)}$, that is, $x\in\cW^{cs}$ for each
$x\in B(\mu)$.

Finally, from Theorem~\ref{mthm:physectional}, there are
finitely many ergodic SRB measures whose basis cover
$\m$-a.e. point of $\cU$. This ensures that
$\m\big(\cU\setminus\cW^{cs}\big)=0$, and completes the proof of
the measure-theoretic part of the statement of
Theorem~\ref{mthm:topbasinsectional}.


\section{Statistical stability of sectional-hyperbolic
  attracting sets}
\label{sec:statist-stabil-secti}

Statistical stability is essentially a consequence of the
existence of finitely many physical measures whose basins
cover $\m$-a.e points of the trapping region together with
recent results from \cite{PaYaYa} on robust entropy
expansiveness of sectional hyperbolic attractors on their
trapping regions. We recall some relevant notions in what
follows to be able to present a proof of
Theorem~\ref{mthm:statstability} in
Subsection~\ref{sec:statist-stabil}.

\subsection{Entropy expansiveness}
\label{sec:entropy-expansiveness}

Let $g : M \to M$ be a continuous map and $K$ a not
necessarily invariant subset of $M$. For $\epsilon > 0$ and
$n \ge 1$, the $(\epsilon,n)$-dynamical ball around
$x \in M$ is 
$
  B(x,\epsilon,n)
  =
  \{y \in M: d(g^jx, g^jy) < \epsilon, \quad \forall 0 \le j < n\}.
  $
  A subset $E \subset M$ is a $(n, \epsilon)$-generator for
  $K$ if, given $x \in K$, there exists $y \in E$ so that
  $d(g^ix, g^iy) < \epsilon$ for each $0 \le i<
  n$. Equivalently, the dynamical ball
  $\{B(y,\epsilon,n): y\in E\}$ are an open cover of $K$.

Let $r_n(K,\epsilon)$ be the cardinality of the smallest
$(n,\epsilon)$-generator for $K$ and
$
  r(K, \epsilon)
  =
  \limsup_{n\to\infty}\frac1n\log r_n(K, \epsilon).
$
The \emph{topological entropy of $g$ on $K$} is given by
\begin{align*}
h_{top}(g, K) = \lim_{\varepsilon\to 0}r(K, \varepsilon),
\end{align*}
and the \emph{topological entropy of $g$} is defined by
$h_{top}(g) = h_{top}(g, M)$.

For $x \in M$ and $\epsilon > 0$ we define the
\emph{two-sided $\epsilon$-dynamical ball at $x$} as
$B(x, \epsilon,\infty) = \{y : d(g^nx, g^ny) < \epsilon\,
\forall n \in \ZZ\}$ and say that \emph{$g$ is
  $\epsilon$-entropy expansive} if all these infinite
dynamical balls have zero topological entropy, that is,
$
\sup_{x\in M} h_{top}\big(g, B(x, \varepsilon,\infty)\big) = 0.
$

\subsection{Upper semicontinuity of metric entropy}
\label{sec:upper-semicont-metri}

Let $\mu$ be a $g$-invariant measure and $\cP$ a finite $\mu\bmod0$
measurable partition. The \emph{metric entropy of $\mu$
with respect to the partition $\cP$} is given by
\begin{align*}
  h_\mu(g,\cP)
  =
  \inf_{n\ge1} \frac1n H_\mu(\cP^n)
  \quad\text{where}\quad
  H_\mu(\cP^n)=\sum_{B\in\cP^{n-1}}-\mu(B) \log \mu(B)
\end{align*}
and $\cP^n$ is the $n$th dynamical refinement of
$\cP$:
$ \cP^n = \cP \vee g^{-1}\cP \vee \cdots \vee
g^{-(n-1)}\cP$. The \emph{metric entropy of $\mu$} is
$ h_{\mu}(g) = \sup_{\cP}h_{\mu}(g,\cP) $ and the supremum is
taken over all finite measurable partitions.


If $g$ is $\epsilon$-entropy expansive, then \emph{every
  finite partition $\cP$ with $\diam \cP<\varepsilon$ is
  generating}, that is, it satisfies
$h_{\mu}(g) = h_\mu(g,\cP)$ for all $\mu\mu\in\M_1^g$, where
$\M_1^g$ is the family of all $g$-invariant probability;
measures see e.g. \cite{Bowen72}.

\emph{The metric entropy of a vector field} is the metric
entropy of the time-one map of its induced flow. \emph{A
  vector field is $\epsilon$-entropy expansive} if the
time-one map of its induced flow is $\epsilon$-entropy
expansive.

Entropy expansiveness is a sufficient condition to ensure
upper semicontinuity of the entropy map
$\mu\in\M_1^g\mapsto h_\mu(g)$, as follows.

\begin{lemma}\cite{Bowen72}\label{le:BowenVar}
  If $G$ is entropy expansive, then the metric entropy
  function is upper semicontinuous.
\end{lemma}

\subsection{Equilibrium states and physical measures}
\label{sec:equilibr-states-phys}

Since the family $\M_1^G$ of $G$-invariant probability
measures is compact in the weak$^*$ topology, then for
entropy expansive vector fields there exist some measure
which maximizes the function
$\mu\in\M_1^G\mapsto h_\mu(G)+\int\psi\,d\mu$ for any given
continuous function $\psi:M\to\RR$, known as an
\emph{equilibrium state for $\psi,G$}.


In order to use equilibrium states to
obtain statistical stability, we relate equilibrium states
for the potencial $\psi=\log|\det D\phi_1\mid E^{cu}|$ with
physical measures in the same way as for hyperbolic
attracting sets; see e.g. \cite{BR75}.

\begin{theorem}
  \label{thm:sectional-SRB-physical}
  Let $\Lambda$ be a sectional-hyperbolic attracting set for
  a H\"older-$C^1$ vector field $G$ with the open subset $U$ as
  trapping region. Then
  \begin{enumerate}
  \item Each $G$-invariant ergodic probability
    measure $\mu$ supported in $\Lambda$ the following are
    equivalent
    \begin{enumerate}
    \item
      $h_\mu(\phi_1)=\int\psi\,d\mu>0$;
    \item $\mu$ is a $SRB$ measure, that is, admits an
      absolutely continuous disintegration along unstable
      manifolds;
    \item $\mu$ is a physical measure, i.e., its basin
      $B(\mu)$ has positive Lebesgue measure.
    \end{enumerate}
  \item In addition, the family $\EE$ of all $G$-invariant
    probability measures which satisfy item (a) above is the
    convex hull
    $ \EE=\{\sum_{i=1}^k t_i \mu_i : \sum_i t_i=1; 0\le
    t_i\le1, i=1,\dots,k\}.  $
  \end{enumerate}
\end{theorem}

We recall that from sectional hyperbolicity together with
Ruelle's Inequality \cite{Ru78} we have
$h_\nu(\phi_1)\le \int \psi\,d\nu$ for all
$\nu\in\M_1^G$. Hence, the set $\EE$ defined above is formed
by equilibrium states for $-\psi,G$.  The proof of
Theorem~\ref{thm:sectional-SRB-physical} can be found in
\cite[Section 2.3]{ArSzTr} where the same properties were
stated and proved in the $d_{cu}=2$ setting
(singular-hyperbolic attracting sets). However, the proof
presented there also holds in the present setting without
change.



\subsection{Statistical stability}
\label{sec:statist-stabil}

Here we prove Theorem~\ref{mthm:statstability}.

We consider vector fields $G$ on a subset $\cU$ of
$\X^{1+}(M)$ with a trapping region $U$ of a sectional
hyperbolic attracting set
$\Lambda_G=\Lambda_G(U)=\cap_{t>0}\phi_t^G(U)$ so that each
$G\in\cU$ is $\epsilon$-entropy expansive. Then the map
$\cU\to\K(U), G\mapsto\Lambda_G(U)$ is continuous, where
$\K(U)$ is the family of compact subsets of $U$ with the
Hausdorff distance between compact subsets $K,L\subset U$ of
a metric space given by (see e.g. \cite{falconer1990})
\begin{align*}
  d_H(K,L)=\inf\{r>0 : K\subset B(L,r) \quad\text{and}\quad
  L\subset B(K,r)\}.
\end{align*}
\begin{lemma}\cite[Lemma 2.3]{AraPac2010}
  \label{le:upper-semicont-maximal}
  For every $\epsilon>0$ there is a neighborhood ${\V}$ of
  $G$ in $\X^1(M)$ such that
  $\Lambda_Y(U)\subset B(\Lambda_G(U),\epsilon)$ and
  $\Lambda_G(U)\subset B(\Lambda_Y(U),\epsilon)$ for all
  $Y\in{\V}$.
\end{lemma}
Moreover the map $\nu\in\M\mapsto\supp\nu\in\K(M)$ is also
continuous, where $\M$ is the family of probability measures
in $M$ with the weak$^*$ topology. In addition, the
domination of the splitting $E_\Lambda^s\oplus
E_\Lambda^{cu}$ implies its continuity for nearby vector
fields; see e.g. \cite[Appendix B.1]{BDV2004}. 

For any fixed $G\in\cU$ and any sequence $G_n\in\cU$ such
that $\| G_n - G \|_{C^1}\to0$ when $n\nearrow\infty$, we
let $\mu_n\in\M_1^{G_n}$ be equilibrium states for
$\psi_n,G_n, n\ge1$, where
$\psi_n=\psi_{G_n}=\log\big|\det D\phi_1^{G_n}\mid
E^{cu}_{\Lambda_{G_n}(U)}\big|$, and $\mu$ be a weak$^*$
limit point of $(\mu_n)_{n\ge1}$. We assume that
$\mu=\lim\mu_n$ restricting to a subsequence if necessary.
Since the splitting
$E^s_{\Lambda_{G_n}(U)}\oplus E^{cu}_{\Lambda_{G_n}(U)}$ is
continuous, we can deal with its continuous extension
$E^s_n\oplus E^{cu}_n$ to define $\psi_n$ on the whole of
$M$.

The continuity of dominated splittings for nearby vector
fields means that for each $\xi>0$ there exists $N\ge1$ and
a neighborhood $V$ of $\supp\mu$ so that
\begin{align*}
  \supp\mu_n\subset V
  \qand
  \dist(E^*_{n, x} ; E^*_{\Lambda_{G}(U), x})<\xi, \quad x\in
  V; *=s,cu, \forall n>N;
\end{align*}
where the distance $\dist(E,F)$ between two subspaces $E,F$
of $T_xM$ is defined to be
\begin{align*}
  \dist(E,F):=\max\big\{\sup_{\|v\|=1,v\in E}\dist(v,F),
\sup_{\|v\|=1,v\in F}\dist(v,E)\big\};
\end{align*}
and $\dist(v,H):=\min_{w\in E}\| v - w \|$ for each subspace
$H$ of $T_xM$ and any $x\in M$.

Moreover, since $D\phi^{G_n}_1(x)$ converges to
$D\phi_1^{G}(x)$ uniformly in $x$ when $n\nearrow\infty$,
then $\psi_n\to\psi=\psi_G$ uniformly by definition of the
$C^1$ topology, in the following sense: for any given
$\xi>0$ there is $N\ge1$ and a neighborhood $V$ of
$\supp\mu$ so that $|\psi_n(x)-\psi(x)|<\xi$ for all
$x\in V$ and each $n>N$.



\begin{proof}[Proof of Theorem~\ref{mthm:statstability}]
  Using the compactness of the manifold $M$, we construct a
  finite open cover $\{B(x_i,\delta): i=1,\dots,k\}$ for
  some $2\delta<\epsilon$ such that
  $\mu(B(x,\delta))=0$, $i=1,\dots, k$; and obtain the
  partition $\cP=\bigvee_{i=1}^kB(x_i,\epsilon/2)$ with
  diameter smaller than $\epsilon$ and the boundaries of
  each atom with zero $\mu$-measure.  Hence, for each
  $k\ge1$, we have that $\mu(\partial\cP^k)=0$ since by
  continuity we have
  \begin{align*}
    \partial \cP^n
    &\subset
      \partial \cP \cup
      \partial(\phi_{-1}\cP)\cup\dots\cup\partial(\phi_{-k+1}\cP)
      \subset
      \partial \cP \cup
      \phi_{-1}\partial\cP\cup\dots\cup \phi_{-k+1}\partial\cP.
  \end{align*}
  Now for each fixed $k\ge1$ we find
  \begin{align*}
    0
    =
    \limsup_n \big(h_{\mu_n}(G_n)+\int \psi_n\,d\mu_n\big)
    \le
    \limsup_n \big(\frac1k H_{\mu_n}(\cP_n^k) +\int \psi_n\,d\mu_n\big)
  \end{align*}
  where $\cP_n^k=\bigvee_{i=0}^{k-1}\phi_{-i}^{G_n}\cP$ and
  $(\phi_t^{G_n})_t$ is the flow induced by $G_n$.

  \begin{lemma}\label{le:limsup}
    For each fixed $k\ge1$ we have
    $ \limsup_n \frac1k H_{\mu_n}(\cP_n^k) \le \frac1k
    H_{\mu}(\cP^k) $ where
    $\cP^k=\bigvee_{i=0}^{k-1}\phi_{-i}^{G}\cP$.
  \end{lemma}

  Assuming the lemma, since $k\ge1$ is arbitrary and
  (possibly taking a subsequence) we have $\mu_n\to\mu$ in
  the weak$^*$ topology, we have
  \begin{align*}
    \big|\int\psi_n\,d\mu_n-\int\psi\,d\mu\big|
    \le
    \big|\int(\psi_n-\psi)\,d\mu_n\big|
    +
    \big|\int\psi\,d\mu_n-\int\psi\,d\mu\big|
    \xrightarrow[n\to\infty]{}0.
  \end{align*}
  Consequently, we deduce that
  \begin{align*}
    0\le\inf_{k\ge1}\big(\frac1k
    H_\mu(\cP^k)-\int\psi\,d\mu\big)
    =h_\mu(G)-\int\psi\,d\mu
    \le0
  \end{align*}
  and so $\mu$ achieves the maximum of
  $\mu\in\M_1^G\mapsto h_\mu(G)-\int\psi\,d\mu$. From
  Theorem~\ref{thm:sectional-SRB-physical} we have that
  $\mu$ is a convex linear combination of the finitely many
  ergodic physical measures supported in $\Lambda_G(U)$
  provided by Theorem~\ref{mthm:physectional}.
\end{proof}

To complete the proof of Theorem~\ref{mthm:statstability} we
present the proof of the lemma.

 \begin{proof}[Proof of Lemma~\ref{le:limsup}]
   Observe that
   $\sup_{|t|<k}d\big(\phi_t^{G_n}(x),\phi_t^{G}(x))
   \xrightarrow[n\to\infty]{u}0$ for all fixed $k\ge1$ and
   uniformly in $x\in M$. Moreover, we may assume without
   loss of generality that each $P\in\cP$ has non-empty
   interior by construction.

   Thus for each $\delta>0$ and atom $Q\in\cP$ there exists
   $N=N(\delta,Q)\in\ZZ^+$ such that for all $n\ge N$ and
   $0\le t\le k$
   \begin{itemize}
   \item $\phi^{G}_{-t}(Q)\cap \phi_{-t}^{G_n} (Q) \neq\emptyset$ and
     $\phi_{-t}^{G_n}(Q)\subset B_\delta(\phi_{-t}^G(Q))$; and
   \item $\mu(\partial B_\delta(Q))=0$;
   \end{itemize}
   where $B_\delta(Q)=\cup_{x\in Q}B(x,\delta)$ is the
   $\delta$-neighborhood of the set $Q$. Let
   $N(\delta,\cP^k)=\max_{Q\in\cP^k} N(\delta,Q)$ be chosen
   to satisfy the previous relations for all $Q\in\cP^k$
   simultaneously.

   For $\omega>0$ let $\zeta>0$ be such that
   \begin{align*}
     |t_i-s_i|<\zeta, t_i,s_i\in\RR, i=1,\dots,k
     \implies
     \sum_{i=1}^k-x_i\log x_i<\omega;
   \end{align*}
   and, for each $\delta>0$, let
   $L=L(\zeta,\delta,\cP^k)$ be such that
   $\mu(\partial B(Q,\delta))=0, \forall Q\in\cP^k$ and
   \begin{align*}
     n\ge L, Q\in\cP^k
     \implies \mu_n(B_\delta(Q))\le \mu(B_\delta(Q))+\frac{\zeta}2.
   \end{align*}
   Since $\mu(\partial\cP^k)=0$, let $\delta_0$ be such that
   $\mu(B_\delta(Q))\le\mu(Q)+\zeta/2$ for all
   $Q\in\cP^k$.

   Now, we take $0<\delta<\delta_0$ in the previous choices,
   and for $n\ge L(\zeta,\delta,\cP^k)+N(\delta,\cP^k)$
   we have for each $Q_n\in\cP^k_n$ that there exists
   $Q\in\cP^k$ so that
   \begin{align*}
     Q_n\subset B_\delta(Q)
     \qand
     \mu_n(Q_n)
     \le
     \mu_n(B_\delta(Q))
     \le
     \mu(B_\delta(Q))+\frac{\zeta}2
     \le
     \mu(Q)+\zeta
   \end{align*}
   which ensures by the choice of the pair
   $(\zeta,\omega)$ that
   \begin{align*}
     \frac1k H_{\mu_n}(\cP^k_n)
     \le
     \frac1k \big(H_{\mu}(\cP^k)+\omega \big)
     \le
     \frac1k H_{\mu}(\cP^k)+\frac{\omega}k
   \end{align*}
   for all big enough $n$ depending on $\omega$. Since
   $\omega>0$ is arbitrary, this shows that
   \begin{align*}
     \limsup_{n\to\infty}\frac1k H_{\mu_n}(\cP^k_n)
     \le
     \frac1k H_{\mu}(\cP^k)
   \end{align*}
   and completes the proof of the lemma
 \end{proof}
 



\def\cprime{$'$}
 
\bibliographystyle{abbrv}


\end{document}